\documentclass[12pt]{article}
\usepackage{amsthm,amsmath,amssymb}
\usepackage{fullpage}
\usepackage{float}
\usepackage{tikz}
\usetikzlibrary{calc}
\usepackage[hidelinks]{hyperref}
\hypersetup{
	pdftitle={Second largest eigenvalue thresholds for the vertex-connectivity of regular graphs},
	pdfauthor={Suil O and Lukai Sui}
}

\newtheorem{thm}{Theorem}[section]
\newtheorem{lem}[thm]{Lemma}
\newtheorem{prop}[thm]{Proposition}
\newtheorem{cor}[thm]{Corollary}

\theoremstyle{remark}
\newtheorem{rem}[thm]{Remark}
\theoremstyle{plain}

\newcommand{\Qt}{\widetilde{Q}}
\newcommand{\psym}{\psi}

\tikzset{
	blob/.style={draw=black!55,dashed,rounded corners=10pt,line width=0.8pt},
	clq/.style={draw=black!55,fill=black!5,line width=0.8pt},
	vtx/.style={circle,draw=black,fill=black!10,minimum size=5mm,inner sep=0pt,font=\scriptsize},
	cut/.style={circle,draw=black,fill=black!25,minimum size=5.5mm,inner sep=0pt,font=\scriptsize},
	ce/.style={black!40,line width=0.4pt},
	se/.style={black!75,line width=0.8pt},
	lbl/.style={font=\footnotesize},
	tag/.style={font=\scriptsize,black!70}
}

\begin{document}
	
	\title{Second largest eigenvalue thresholds for the vertex-connectivity of
		regular graphs}
	\author{%
		Suil O$^{*,\dagger}$
		\and
		Lukai Sui$^{*}$
	}
	
	\maketitle
	
	{\renewcommand{\thefootnote}{\fnsymbol{footnote}}%
		\footnotetext[1]{Department of Applied Mathematics and Statistics, The State University of New York, Korea, Incheon, 21985. Email addresses: suil.o@sunykorea.ac.kr (S. O), lukai.sui@stonybrook.edu (L. Sui).}%
		\footnotetext[2]{Corresponding author. Research supported by the National Research Foundation of Korea (NRF) grant funded by the Korea government (MSIT) No. RS-2025-23523950.}}%
	\setcounter{footnote}{0}
	
	\begin{abstract}
		\noindent
		Let $\lambda_2(G)$ and $\kappa(G)$ denote the second largest adjacency eigenvalue and the
		vertex-connectivity of a graph $G$. For a non-complete $r$-regular graph $G$, Fiedler's
		inequality $\lambda_2(G)\ge r-\kappa(G)$ shows that $\lambda_2(G)<r-t$ forces
		$\kappa(G)\ge t+1$. The best such bound is known for $t=1$, by Zhang and Wang, and for
		$t=2$ with $r\ge4$, by Zhang. In this
		paper, for every $r\ge3$ and every integer $t$ with $2\le t\le r-1$, we give an upper bound
		on $\lambda_2(G)$ that forces an $r$-regular graph $G$ to satisfy $\kappa(G)\ge t+1$, and we
		determine exactly when this bound is best possible.
		For $\frac{r+2}{2}\le t\le r-1$, Fiedler's bound $r-t$ is already best possible. For
		$2\le t<\frac{r+2}{2}$, it is replaced by one of two explicit algebraic numbers, and the sign
		of an explicit polynomial $W(r,t)$, which never vanishes, decides which one applies. Each of
		the two numbers is attained exactly when $(r,t)$ satisfies an explicit arithmetic condition,
		and every extremal graph then consists of two cliques and an independent cut of size $t$.
		In this range our bound also improves a bound of Zhang that holds for every such $t$.
		
		
	\end{abstract}
	
	\section{Introduction}
	
	All graphs in this paper are finite and simple. For a graph $G$ we write $A(G)$ for its
	adjacency matrix and $\lambda_1(G)\ge\lambda_2(G)\ge\cdots\ge\lambda_{|V(G)|}(G)$ for its
	eigenvalues, and we write $\lambda_{\min}(G)=\lambda_{|V(G)|}(G)$. The Laplacian matrix
	of $G$ is
	$L(G)=D(G)-A(G)$, where $D(G)$ is the diagonal matrix of degrees, and its eigenvalues are
	denoted $\mu_1(G)\le\mu_2(G)\le\cdots$. If $G$ is $r$-regular, then
	$\mu_i(G)=r-\lambda_i(G)$ for every $i$. For $S\subseteq V(G)$ we write $G[S]$ for the
	subgraph induced by $S$ and $G-S$ for the subgraph obtained by deleting the vertices of
	$S$; for disjoint $S,T\subseteq V(G)$ we write $e(S,T)$ for the number of edges joining
	$S$ and $T$. A graph $G$ is \emph{$t$-vertex-connected} if $|V(G)|>t$ and $G-S$ is
	connected for every $S\subseteq V(G)$ with $|S|<t$, and the \emph{vertex-connectivity}
	$\kappa(G)$ is the largest such $t$. The \emph{join} $G\vee H$ is obtained from disjoint
	copies of $G$ and $H$ by joining every vertex of $G$ to every vertex of $H$, and $G\cup H$
	denotes the disjoint union.
	
	In 1973, Fiedler~\cite{F} proved that $\mu_2(G)\le\kappa(G)$ for every non-complete graph
	$G$, and he called $\mu_2(G)$ the \emph{algebraic connectivity} of $G$. For a non-complete
	$r$-regular graph this reads
	\begin{equation}\label{eq:fiedler}
		\lambda_2(G)\;\ge\;r-\kappa(G).
	\end{equation}
	The only complete $r$-regular graph is $K_{r+1}$, for which $\lambda_2=-1$ and
	$\kappa=r$.
	Fiedler's inequality has stimulated a great deal of work. Kirkland, Molitierno, Neumann
	and Shader~\cite{KMNS} characterized the graphs attaining equality in
	$\mu_2(G)\le\kappa(G)$, and the first author~\cite{O2} extended Fiedler's bound to
	multigraphs and showed that, unlike in the simple case, no upper bound on $\lambda_2(G)$
	forces a regular multigraph $G$ to be $3$-vertex-connected.
	
	The edge analogue of this circle of problems is by now well understood. Writing
	$\kappa'(G)$ for the edge-connectivity, Cioab\u{a}~\cite{C} gave sharp upper bounds on
	$\lambda_2(G)$ in an $r$-regular graph forcing $\kappa'(G)\ge t+1$ for $t=1,2$, and the
	question was settled for all $t\ge3$ by the first author, Park, Park and Zhang~\cite{OPPZ},
	where a single closed form, with two cases according to the parity of $t$, serves every
	$t\ge3$ and is attained by a single family of graphs. The present paper treats the vertex-connectivity version of~\cite{OPPZ}, about which
	much less is known. The question of determining, for each $t$, the best upper bound on $\lambda_2(G)$ forcing
	$\kappa(G)\ge t+1$ was raised by Mohar and the first author, and also by
	Cioab\u{a}, in private communication, and it is recorded in~\cite{ZW},
	where Zhang and Wang solved the first case $t=1$ and characterized all extremal graphs.
	Zhang~\cite{Z} solved the case $t=2$ for $r\ge4$, and in the same paper showed that
	\begin{equation}\label{eq:zhang}
		\zeta(r,t)=r-\frac{t}{2}\cdot\frac{r+2}{r+3-t}
	\end{equation}
	is a threshold for every $t$ with $t<\frac{r+2}{2}$. Abiad, Brimkov,
	Mart\'{\i}nez-Rivera, the first author and Zhang~\cite{ABMOZ} gave a bound depending on
	the order of $G$.
	
	We use the following terminology throughout. Let $r\ge3$ and let $t$ be an integer with
	$1\le t\le r-1$; values $t\ge r$ are excluded because $\kappa(G)\le r$ for every
	$r$-regular graph $G$. A real number $c$ is a \emph{threshold} for $(r,t)$ if every
	connected $r$-regular graph $G$ with $\lambda_2(G)<c$ satisfies $\kappa(G)\ge t+1$, and
	such a $c$ is \emph{sharp} if in addition some connected $r$-regular graph $G$ with
	$\kappa(G)\le t$ satisfies $\lambda_2(G)=c$; a sharp threshold cannot be replaced by any
	larger number, and we call a graph witnessing its sharpness \emph{extremal}.
	Inequality~\eqref{eq:fiedler} says that $r-t$ is a threshold for every such $t$: a
	connected $r$-regular graph $G$ with $\kappa(G)\le t\le r-1$ is not complete, so it
	satisfies $\lambda_2(G)\ge r-\kappa(G)\ge r-t$. By~\cite{Z,ZW}, a sharp threshold is
	known for $t=1$ and for $t=2$ with $r\ge4$.
	
	Our main result gives a threshold for every pair $(r,t)$ with $t\ge2$ and decides exactly
	when it is sharp. For $t<\frac{r+2}{2}$ it involves two explicit numbers. Let
	\begin{equation}\label{eq:psi}
		\psym(r,t)=\frac{r-1+\sqrt{(r+1)^2-2rt}}{2},
	\end{equation}
	and let $\psi_{\mathrm{bd}}(r,t)$ be the largest root of
	\begin{equation}\label{eq:psibd}
		x^{2}+(t-\eta_0)x-\eta_0-(t-1)(r-t)=0,\qquad
		\eta_0=\frac{r-1+\sqrt{(r+1)^{2}-4t(t-1)}}{2}.
	\end{equation}
	Which of the two applies is decided by the sign of the polynomial
	\begin{equation}\label{eq:W}
		\begin{aligned}
			W(r,t)=\;&r^{3}\bigl(12t^{2}-20t+8\bigr)+r^{2}\bigl(-33t^{3}+88t^{2}-80t+24\bigr)\\
			&+r\bigl(12t^{4}-56t^{3}+104t^{2}-84t+24\bigr)-4t^{5}+16t^{4}-32t^{3}+40t^{2}-28t+8 .
		\end{aligned}
	\end{equation}
	
	\begin{thm}\label{thm:main}
		Let $r\ge3$ and let $t$ be an integer with $2\le t\le r-1$.
		\begin{enumerate}
			\item[$(i)$] If $\frac{r+2}{2}\le t\le r-1$, then $r-t$ is a sharp threshold for $(r,t)$.
			
			\item[$(ii)$] If $t<\frac{r+2}{2}$ and $W(r,t)>0$, then every connected $r$-regular graph
			$G$ with $\kappa(G)\le t$ satisfies
			\begin{equation}\label{eq:main}
				\lambda_2(G)\;\ge\;\psym(r,t),
			\end{equation}
			so $\psym(r,t)$ is a threshold for $(r,t)$. It is sharp if and only if $r$ is even and
			there is an integer $k$ with $1\le k\le\frac r2$ such that
			\[
			r\mid 2k(k-1)\qquad\text{and}\qquad t=\frac{2k(r+1-k)}{r} .
			\]
			In that case every extremal graph consists of two cliques of order $r+1-k$ and an
			independent cut of size $t$ each of whose vertices has $\frac r2$ neighbors in each
			clique; it has $2(r+1-k)+t$ vertices and $\lambda_2=r-k$.
			
			\item[$(iii)$] If $t<\frac{r+2}{2}$ and $W(r,t)<0$, then $\psi_{\mathrm{bd}}(r,t)$ is a
			threshold for $(r,t)$, and $\psi_{\mathrm{bd}}(r,t)<\psym(r,t)$. It is sharp if and only if
			\[
			(r+1)^{2}-4t(t-1)=D^{2}
			\]
			for an integer $D$ with $1\le D\le r-1$. In that case every extremal graph consists of
			two cliques of orders $r+1-t$ and $\frac{r+1+D}{2}$ and an independent cut of size $t$
			joined to every vertex of the smaller clique; it has $\frac{3(r+1)+D}{2}$ vertices.
		\end{enumerate}
		Moreover, $W(r,t)\ne0$ for all integers $r$ and $t$ with $2\le t<r$
		(Proposition~\ref{prop:nonzero}), so $(ii)$ and $(iii)$
		together cover every $t$ with $2\le t<\frac{r+2}{2}$.
	\end{thm}
	
	The two numbers in $(ii)$ and $(iii)$ are the two endpoints of a one-parameter family of
	lower bounds (Proposition~\ref{prop:arc}), and $W(r,t)$ decides which endpoint gives the
	smaller value (Proposition~\ref{prop:endpoint}). Asymptotically, $(ii)$ governs the range
	$t$ up to about $0.4184\,r$ and $(iii)$ the rest of the range $t<\frac{r+2}{2}$; see Remark~\ref{rem:asym}.
	Extremal graphs for $(i)$ are constructed in Subsection~\ref{sec:large}: when $rt$ is even
	they are obtained by joining two disjoint cliques of equal order to a regular graph on $t$
	vertices, and when $rt$ is odd the two sides have consecutive orders and only the smaller
	one is complete.

	Both thresholds improve $\zeta(r,t)$. For $\psym(r,t)$ this is a short computation:
	writing $N=r^{2}-2rt+4r+3-3t$, the inequality $\psym(r,t)>\zeta(r,t)$ is equivalent
	to $(r+3-t)\sqrt{(r+1)^{2}-2rt}>N$, and
	\[
	\bigl((r+1)^{2}-2rt\bigr)(r+3-t)^{2}-N^{2}
	=t\bigl(r^{2}t-2rt^{2}+2rt+4r-8t+12\bigr)
	=t\bigl(tu^{2}+2t^{2}u+4u+2t^{2}-t+8\bigr),
	\]
	where $r=2t-1+u$; the range $t<\frac{r+2}{2}$ is exactly $u\ge0$, and the last expression is
	positive there. The same conclusion holds for $\psi_{\mathrm{bd}}(r,t)$. Denote by $h(x)$ the
	left side of~\eqref{eq:psibd}. It opens upwards and $h(0)<0$, so one root is negative,
	and since $\zeta=\zeta(r,t)>0$ it suffices to show $h(\zeta)<0$. Writing
	$D_0=\sqrt{(r+1)^{2}-4t(t-1)}$, so that $\eta_0=\frac{r-1+D_0}{2}$, we have
	$2h(\zeta)=P-D_0(\zeta+1)$, where
	$P=2\zeta^{2}+2t\zeta-2(t-1)(r-t)-(r-1)(\zeta+1)$. Here $\zeta+1>0$, because
	$\zeta-(r-t)=\frac{t(r-2t+4)}{2(r+3-t)}>0$. Hence it is enough that
	$D_0^{2}(\zeta+1)^{2}>P^{2}$, since then $D_0(\zeta+1)>|P|\ge P$. A computation gives
	\[
	4(r+3-t)^{4}\bigl(D_0^{2}(\zeta+1)^{2}-P^{2}\bigr)=t\,F(r,t)
	\]
	for a polynomial $F$ of degree $7$ with integer coefficients. After the substitution
	$r=2t-1+u$, every coefficient of $F$ as a polynomial in $u$ and $t$ is positive except that
	of $t$, which is $-144$, while the coefficient of $t^{2}$ is $1432$. Since $t\ge2$ and
	$u\ge0$, we get $F>0$, and therefore $\psi_{\mathrm{bd}}(r,t)>\zeta(r,t)$.
	
	Taking $k=1$ in $(ii)$ makes the divisibility condition vacuous, and
	$W(r,2)=16\bigl(r^{3}-3r^{2}+r-1\bigr)>0$ for $r\ge3$. Hence, for every even $r\ge4$, the
	number $r-1$ is a sharp threshold for $(r,2)$, attained by the graph obtained from
	$K_r\cup K_r$ and an independent set of size $2$ by joining each vertex of the independent
	set to $\frac r2$ vertices of each clique; see Figure~\ref{fig:t2}. This recovers the even
	case of the theorem of Zhang~\cite{Z} cited above; for odd $r\ge5$ he showed that the sharp
	threshold for $(r,2)$ exceeds $r-1$, so $\psym(r,2)=r-1$ is then not sharp, in agreement with the
	parity condition in $(ii)$. The sign condition in
	$(ii)$ cannot be dropped: for $(r,t)=(12,6)$ the arithmetic conditions of $(ii)$ hold and the
	graph they describe has $\lambda_2=\psym(12,6)$, but $W(12,6)<0$ and $\psym(12,6)$ is not a
	threshold (Remark~\ref{rem:list}). In contrast to $(ii)$, the condition in $(iii)$ does not
	require $r$ to be even; for instance, it holds for $(r,t)=(7,4)$.
	
	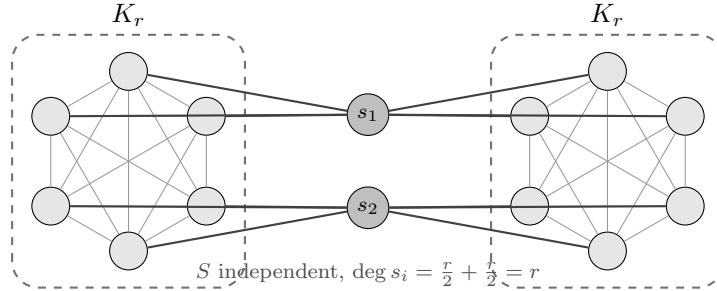
\begin{figure}[H]
		\centering
		\begin{tikzpicture}[scale=0.88]
			\node[lbl] at (-3.6,2.55) {$K_r$};
			\node[lbl] at ( 3.6,2.55) {$K_r$};
			\foreach \i/\ang in {1/90,2/150,3/210,4/270,5/330,6/30}
			\node[vtx] (a\i) at ($(-3.6,0.35)+(\ang:1.35)$) {};
			\foreach \i in {1,...,6} \foreach \j in {1,...,6}
			{ \ifnum\i<\j \draw[ce] (a\i)--(a\j); \fi }
			\draw[blob] (-5.35,-1.55) rectangle (-1.85,2.25);
			\foreach \i/\ang in {1/90,2/30,3/330,4/270,5/210,6/150}
			\node[vtx] (b\i) at ($(3.6,0.35)+(\ang:1.35)$) {};
			\foreach \i in {1,...,6} \foreach \j in {1,...,6}
			{ \ifnum\i<\j \draw[ce] (b\i)--(b\j); \fi }
			\draw[blob] (1.85,-1.55) rectangle (5.35,2.25);
			\node[cut] (s1) at (0, 1.05) {$s_1$};
			\node[cut] (s2) at (0,-0.35) {$s_2$};
			\foreach \i in {a1,a2,a6} \draw[se] (\i)--(s1);
			\foreach \i in {a3,a4,a5} \draw[se] (\i)--(s2);
			\foreach \i in {b1,b2,b6} \draw[se] (\i)--(s1);
			\foreach \i in {b3,b4,b5} \draw[se] (\i)--(s2);
			\node[tag,align=center] at (0,-1.35) {$S$ independent, $\deg s_i=\frac r2+\frac r2=r$};
		\end{tikzpicture}
		\caption{The extremal graph for $t=2$ and even $r$, drawn for $r=6$; it is the graph
			$H_{r,r/2}$ of~\cite{Z}. Both sides are complete of order $r$, so $\lambda_2=r-1$ by
			Corollary~\ref{cor:diag}.}\label{fig:t2}
	\end{figure}
	
	Although the question is the vertex analogue of~\cite{OPPZ}, the answer has a different
	character, and the vertex case is considerably harder. In the edge case the two sides of a
	minimum edge cut interact only through the $t$ cut edges, so each side is constrained
	separately by a minimum-order condition, and one family of graphs, two copies of the same
	graph on $r+1$ or $r+2$ vertices joined by $t$ edges, is extremal for every $t\ge3$. For a vertex
	cut $S$ the two sides $A$ and $B$ are not adjacent, but they share $S$, whose vertices
	divide their neighbors between $A$ and $B$. In place of two separate conditions there is a
	single budget $e(A,S)+e(B,S)\le rt$ for which the two sides compete, and the problem becomes
	a genuinely two-parameter optimization over the average degrees of $A$ and $B$. The
	admissible pairs form an arc rather than a point, the lower bound is minimized at one of
	its two endpoints, and which endpoint gives the minimum changes with $(r,t)$; this is the
	source of the polynomial $W(r,t)$. Moreover, a graph attaining the bound must divide the
	budget into integer parts, and this requirement, which is harmless in the edge case,
	produces the arithmetic conditions in $(ii)$ and $(iii)$. Accordingly, no single closed form
	serves every $t$, and the shape of the extremal graphs changes with $t$: for $t=1$ neither
	side is complete (Figure~\ref{fig:t1}); in $(ii)$ both sides are complete of equal order
	(Figures~\ref{fig:t2} and~\ref{fig:diag}); in $(iii)$ they are complete of different orders
	(Figure~\ref{fig:bd}); and in $(i)$ the cut is joined to everything outside itself
	(Figure~\ref{fig:large}).
	
	The paper is organized as follows. Section~\ref{sec:def} collects the tools.
	Section~\ref{sec:main} contains the key identity, the reduction to the arc, the comparison
	of its endpoints, and the proof of Theorem~\ref{thm:main}, with part $(i)$ treated in
	Subsection~\ref{sec:large}. Section~\ref{sec:remark} contains concluding remarks. In the figures, dashed regions are the two sides of the cut, shaded circles are
	cut vertices, thin lines are edges inside a side, and thick lines are edges to the cut.
	
	\begin{figure}[H]
		\centering
		\begin{tikzpicture}[scale=0.92]
			\draw[clq,rounded corners=14pt] (-6.1,-1.75) rectangle (-1.5,1.75);
			\node[lbl] at (-3.8,2.1) {$K_{r+1}-\frac{c}{2}K_2$};
			\foreach \i/\y in {1/1.05,2/0.35,3/-0.35,4/-1.05} \node[vtx] (a\i) at (-2.15,\y) {};
			\foreach \i/\y in {1/0.7,2/-0.7} \node[vtx] (ac\i) at (-4.1,\y) {};
			\node[vtx] (ac3) at (-5.4,0) {};
			\foreach \i in {a1,a2,a3,a4} \foreach \j in {ac1,ac2,ac3} \draw[ce] (\i)--(\j);
			\draw[ce] (ac1)--(ac2) (ac1)--(ac3) (ac2)--(ac3);
			\draw[ce] (a1)--(a2) (a3)--(a4);
			\draw[ce] (a1) to[bend right=32] (a3);
			\draw[ce] (a2) to[bend right=32] (a4);
			\node[tag,align=center] at (-3.8,-2.15) {the $c$ deficient vertices};
			\draw[clq,rounded corners=14pt] (1.5,-1.75) rectangle (6.1,1.75);
			\node[lbl] at (3.8,2.1) {$K_{r+1}-\frac{r-c}{2}K_2$};
			\foreach \i/\y in {1/1.05,2/0.35,3/-0.35,4/-1.05} \node[vtx] (b\i) at (2.15,\y) {};
			\foreach \i/\y in {1/0.7,2/-0.7} \node[vtx] (bc\i) at (4.1,\y) {};
			\node[vtx] (bc3) at (5.4,0) {};
			\foreach \i in {b1,b2,b3,b4} \foreach \j in {bc1,bc2,bc3} \draw[ce] (\i)--(\j);
			\draw[ce] (bc1)--(bc2) (bc1)--(bc3) (bc2)--(bc3);
			\draw[ce] (b1)--(b2) (b3)--(b4);
			\draw[ce] (b1) to[bend left=32] (b3);
			\draw[ce] (b2) to[bend left=32] (b4);
			\node[tag,align=center] at (3.8,-2.15) {the $r-c$ deficient vertices};
			\node[cut] (v) at (0,0) {$v$};
			\foreach \i in {a1,a2,a3,a4} \draw[se] (\i)--(v);
			\foreach \i in {b1,b2,b3,b4} \draw[se] (\i)--(v);
			\node[tag] at (0,-2.15) {$\deg v=c+(r-c)=r$};
		\end{tikzpicture}
		\caption{The extremal graphs of Zhang and Wang~\cite{ZW} for $t=1$ and even $r$, with
			$c=2\lfloor r/4\rfloor$. Neither side is complete, so Corollary~\ref{cor:diag} does not
			apply and $\lambda_2$ is a root of a quartic.}\label{fig:t1}
	\end{figure}
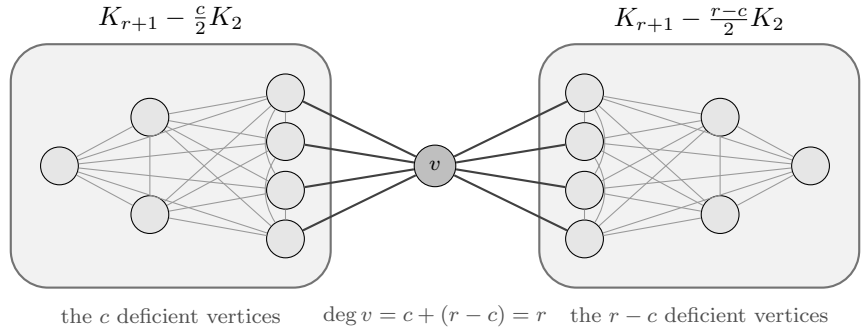
	
	For undefined terms in graph theory, see West~\cite{W}; for the basics of spectral graph
	theory, see Brouwer and Haemers~\cite{BH} or Godsil and Royle~\cite{GR}.
	
	\section{Definitions and Tools}\label{sec:def}
	
	Let $\mathcal P=\{V_1,\dots,V_s\}$ be a partition of $V(G)$. The \emph{quotient matrix}
	$Q$ of $\mathcal P$ is the $s\times s$ matrix with $Q_{i,j}=\frac{e(V_i,V_j)}{|V_i|}$ for
	$i\ne j$ and $Q_{i,i}=\frac{2|E(G[V_i])|}{|V_i|}$. The partition is \emph{equitable} if
	every vertex of $V_i$ has exactly $Q_{i,j}$ neighbors in $V_j$ for all $i,j$; in that case
	every eigenvalue of $Q$ is an eigenvalue of $G$.
	
	\begin{thm}[Quotient Interlacing Theorem, \cite{BH}]\label{thm:interlace}
		The eigenvalues of the quotient matrix of a partition of $V(G)$ interlace the eigenvalues
		of $G$.
	\end{thm}
	
	We shall use the following reduction, which replaces a tridiagonal quotient matrix with
	constant row sums by one of smaller order carrying the same non-Perron spectrum.
	
	\begin{thm}[\cite{OPPZ}]\label{thm:reduce}
		Let
		\[
		M=\begin{pmatrix}
			a_0 & b_0 & & \\
			c_1 & a_1 & \ddots & \\
			& \ddots & \ddots & b_{n-1}\\
			& & c_n & a_n
		\end{pmatrix}
		\]
		be a nonnegative tridiagonal matrix each of whose row sums equals $r$, with eigenvalues
		$\lambda_1\ge\cdots\ge\lambda_{n+1}$. Then the $n\times n$ matrix
		\[
		\widetilde M=\begin{pmatrix}
			r-b_0-c_1 & b_1 & & \\
			c_1 & r-b_1-c_2 & \ddots & \\
			& \ddots & \ddots & b_{n-1}\\
			& & c_{n-1} & r-b_{n-1}-c_n
		\end{pmatrix}
		\]
		has eigenvalues $\lambda_2,\dots,\lambda_{n+1}$. In particular $\lambda_1(M)=r$ and
		$\lambda_2(M)=\lambda_1(\widetilde M)$.
	\end{thm}
	
	We record two standard facts. First, if $G$ is $d$-regular on $n$ vertices, then
	$\overline G$ is $(n-1-d)$-regular and
	\begin{equation}\label{eq:comp}
		\lambda_2(G)=-1-\lambda_{\min}\bigl(\overline G\bigr).
	\end{equation}
	Second, a $d$-regular graph satisfies $\lambda_{\min}\ge-d$, with equality if and only if
	it has a bipartite component. Finally, a $d$-regular graph on $n$ vertices exists if and
	only if $d\le n-1$ and $dn$ is even; when it exists, a circulant realizes it.
	
	\section{Main Results}\label{sec:main}
	
	Except in the last subsection, $G$ is a connected $r$-regular graph with $\kappa(G)\le t$ and
	$2\le t<\frac{r+2}{2}$. A vertex cut of size exactly $t$ exists: take a minimum vertex cut
	$S_{0}$ and a side $A_{0}$ of it, and adjoin to $S_{0}$ any $t-|S_{0}|$ vertices of
	$A_{0}$, which is possible because every vertex of $A_{0}$ has all of its neighbors in
	$A_{0}\cup S_{0}$, so $|A_{0}|\ge r+1-|S_{0}|>t-|S_{0}|$. Fix such a cut $S$ with $|S|=t$
	and a partition $A,B$ of $V(G)\setminus S$ with $e(A,B)=0$, and set
	\[
	a=|A|,\quad b=|B|,\quad K=e(A,S),\quad L=e(B,S),\quad
	\xi=\frac{2\bigl|E(G[A])\bigr|}{a},\quad \eta=\frac{2\bigl|E(G[B])\bigr|}{b}.
	\]
	Since every vertex of $A$ has all of its neighbors in $A\cup S$, counting the edges
	leaving $A$ gives $K=a(r-\xi)$, and likewise $L=b(r-\eta)$. Thus $\xi$ and $\eta$
	are determined by $a,b,K,L$. No regularity of $G[A]$, of $G[B]$, or of the bipartite
	graphs between the parts is assumed.
	
	\subsection{A quadratic identity}
	
	The partition $\{A,S,B\}$ has a nonnegative tridiagonal quotient matrix all of whose row
	sums equal $r$, so Theorems~\ref{thm:interlace} and~\ref{thm:reduce} give
	\begin{equation}\label{eq:Qt}
		\lambda_2(G)\;\ge\;\lambda_1\bigl(\Qt\bigr),
		\qquad
		\Qt=\begin{pmatrix}
			r-K\bigl(\frac1a+\frac1t\bigr) & \frac Lt\\[4pt]
			\frac Kt & r-L\bigl(\frac1t+\frac1b\bigr)
		\end{pmatrix}
		=\begin{pmatrix} q_1 & \frac Lt\\[2pt] \frac Kt & q_2\end{pmatrix}.
	\end{equation}
	
	\begin{lem}\label{lem:key}
		With the notation above,
		\begin{equation}\label{eq:key}
			t\cdot\det\bigl(xI-\Qt\bigr)=t\,(x-\xi)(x-\eta)+(x-\xi)L+(x-\eta)K .
		\end{equation}
		If $\xi\ne\eta$, then equivalently, $\lambda_1(\Qt)$ is the largest root $\lambda$ of
		\[
		\frac{K}{t(\lambda-\xi)}+\frac{L}{t(\lambda-\eta)}=-1 .
		\]
	\end{lem}
	
	\begin{proof}
		Since $K=a(r-\xi)$ we have $q_1=r-\frac Ka-\frac Kt=\xi-\frac Kt$, and likewise
		$q_2=\eta-\frac Lt$. Therefore
		\[
		t\det(xI-\Qt)=t\Bigl(x-\xi+\tfrac Kt\Bigr)\Bigl(x-\eta+\tfrac Lt\Bigr)-\tfrac{KL}{t},
		\]
		and expanding the product gives
		$t(x-\xi)(x-\eta)+(x-\xi)L+(x-\eta)K+\frac{KL}{t}-\frac{KL}{t}$, which
		is~\eqref{eq:key}. Dividing~\eqref{eq:key} by $t(x-\xi)(x-\eta)$ gives the second
		form.
	\end{proof}
	
	The following consequence is the source of the closed expression~\eqref{eq:psi}.
	
	\begin{cor}\label{cor:diag}
		If $\xi=\eta$, then~\eqref{eq:key} factors as
		$t(x-\xi)\bigl[(x-\xi)+\frac{K+L}{t}\bigr]$, so $\lambda_1(\Qt)=\xi$, whatever
		$K$ and $L$ are.
	\end{cor}
	
	The graphs we construct in this section have the following shape, and we record once
	and for all what their spectra look like. Call $G$ a \emph{clique pair over $S$} if $V(G)$ is
	partitioned into $A$, $S$ and $B$ so that $S$ is independent, no edge joins $A$ to $B$,
	both $G[A]$ and $G[B]$ are complete, and the bipartite graphs between $A$ and $S$ and
	between $B$ and $S$ are biregular. We write $N_A$ and $N_B$ for their biadjacency
	matrices, with $A$-degree $k_A=r+1-a$ and $S$-degree $d_A$, and likewise for $B$.
	
	\begin{lem}\label{lem:spec}
		Let $G$ be a connected $r$-regular clique pair over $S$ with $|S|=t$, and let
		\[
		\nu=\lambda_2\bigl(N_A^{\top}N_A\bigr)+\lambda_2\bigl(N_B^{\top}N_B\bigr),
		\qquad \theta=\frac{-1+\sqrt{1+4\nu}}{2}.
		\]
		If $\theta\le\lambda_1(\Qt)$, then $\lambda_2(G)=\lambda_1(\Qt)$. Here
		$N_A^{\top}N_A$ has constant row sums $k_Ad_A$, so that
		$\lambda_2\bigl(N_A^{\top}N_A\bigr)\le k_Ad_A$, with
		$\lambda_2\bigl(N_A^{\top}N_A\bigr)=0$ when every vertex of $A$ is adjacent to
		every vertex of $S$; the same holds for $B$.
	\end{lem}
	
	\begin{proof}
		The partition $\{A,S,B\}$ is equitable, since $G[A]$ and $G[B]$ are complete and the two
		bipartite graphs are biregular. Its quotient matrix $Q$ is nonnegative and irreducible,
		because $G$ is connected, and all its row sums equal $r$; hence $r$ is a simple eigenvalue
		of $Q$, and the eigenvalues of $Q$ are $r$ together with those of $\Qt$ by
		Theorem~\ref{thm:reduce}. In particular $\lambda_1(\Qt)<r$, and $\lambda_1(\Qt)$ is an
		eigenvalue of $G$.
		
		Since the partition is equitable, the space $U$ of vectors summing to zero on each of $A$,
		$S$ and $B$ is invariant under $A(G)$, and every eigenvalue of $G$ is either an eigenvalue
		of $Q$ or an eigenvalue of the restriction of $A(G)$ to $U$. Let $\lambda$ be such a
		restricted eigenvalue, with eigenvector $(x,w,y)$ written along $A$, $S$, $B$. As $G[A]$
		and $G[B]$ are complete and $S$ is independent,
		\[
		-x+N_Aw=\lambda x,\qquad -y+N_Bw=\lambda y,\qquad N_A^{\top}x+N_B^{\top}y=\lambda w .
		\]
		Both $N_A^{\top}N_A$ and $N_B^{\top}N_B$ are positive semidefinite, so $\nu\ge0$ and
		hence $\theta\ge0$. In particular, if $\lambda=-1$, then $\lambda\le\theta$. Otherwise
		$x=\frac{1}{\lambda+1}N_Aw$ and $y=\frac{1}{\lambda+1}N_Bw$, so that
		\[
		\bigl(N_A^{\top}N_A+N_B^{\top}N_B\bigr)w=\lambda(\lambda+1)\,w ,
		\]
		and $w\ne0$, since $w=0$ would force $x=y=0$. Now $N_A^{\top}N_A$ is symmetric and
		nonnegative with constant row sums $k_Ad_A$, so the all-ones vector is a Perron
		eigenvector of it, and likewise for $N_B^{\top}N_B$. As $w$ is orthogonal to the all-ones
		vector, the Rayleigh quotient of each of the two matrices at $w$ is at most its second
		largest eigenvalue, whence $\lambda(\lambda+1)\le\nu$ and therefore $\lambda\le\theta$.
		
		Thus every eigenvalue of $G$ other than $r$ is at most $\max\{\lambda_1(\Qt),\theta\}$,
		and $\lambda_1(\Qt)$ is one of them. If $\theta\le\lambda_1(\Qt)$, then
		$\lambda_2(G)=\lambda_1(\Qt)$.
		
		For the last statement, if every vertex of $A$ is adjacent to every vertex of $S$, then
		$N_A$ is the all-ones matrix and $N_A^{\top}N_A=aJ_t$, whose second largest eigenvalue
		is $0$.
	\end{proof}
	
	\subsection{Reduction to one variable}
	
	Put
	\begin{equation}\label{eq:F}
		\mathcal F(r,t)=\Bigl\{(\xi,\eta)\in\mathbb R^2:\ r-t\le\xi\le\eta\le r,\ \
		(\xi+1)(r-\xi)+(\eta+1)(r-\eta)= rt\Bigr\},
	\end{equation}
	and for $(\xi,\eta)\in\mathcal F(r,t)$ let $\Lambda(\xi,\eta)$ be the largest root of
	\begin{equation}\label{eq:quad}
		t(x-\xi)(x-\eta)+(x-\xi)(\eta+1)(r-\eta)+(x-\eta)(\xi+1)(r-\xi)=0 .
	\end{equation}
	The aim of this subsection is the following.
	
	\begin{prop}\label{prop:arc}
		Let $G$ be a connected $r$-regular graph with $\kappa(G)\le t$, where
		$2\le t<\frac{r+2}{2}$. Then
		\[
		\lambda_2(G)\;\ge\;\min_{(\xi,\eta)\in\mathcal F(r,t)}\Lambda(\xi,\eta).
		\]
		Moreover $\mathcal F(r,t)$ is an arc, and the minimum is attained only at its two
		endpoints, that is, where $\xi=\eta$ or where $\xi=r-t$.
	\end{prop}
	
	The endpoints of the arc are identified by the following lemma, which we prove first.
	
	\begin{lem}\label{lem:concave}
		Put $c=\frac{r+1}{2}$, $p=\xi-\frac{r-1}{2}$, $q=\eta-\frac{r-1}{2}$ and
		$\sigma=p+q$. Then:
		\begin{enumerate}
			\item[$(i)$] $(\xi+1)(r-\xi)=c^{2}-p^{2}$ and $(\eta+1)(r-\eta)=c^{2}-q^{2}$, and the
			constraint in~\eqref{eq:F} is the circle $p^{2}+q^{2}=R^{2}$, where
			$R^{2}=\frac{(r+1)^{2}}{2}-rt>\frac12$;
			\item[$(ii)$] in the variable $y=x-\frac{r-1}{2}$, equation~\eqref{eq:quad} divided by $t$
			becomes
			\[
			y^{2}+(r-\sigma)y+H(\sigma)=0,\qquad
			H(\sigma)=\frac{\sigma^{2}-R^{2}}{2}\Bigl(1+\frac{\sigma}{t}\Bigr)-\frac{c^{2}\sigma}{t},
			\]
			so that $\Lambda$ depends only on $\sigma$;
			\item[$(iii)$] on the feasible arc one has $\sigma>0$ and
			$\Delta''(\sigma)=-2-\frac{12\sigma}{t}<0$, where
			$\Delta(\sigma)=(r-\sigma)^{2}-4H(\sigma)$; consequently
			$\Lambda(\sigma)=\frac{(\sigma-r)+\sqrt{\Delta(\sigma)}}{2}+\frac{r-1}{2}$ is
			strictly concave in $\sigma$.
		\end{enumerate}
		In particular the minimum of $\Lambda$ over $\mathcal F(r,t)$ is attained only at an
		endpoint of the $\sigma$-interval, that is, either at $\xi=\eta$ or at $\xi=r-t$.
	\end{lem}
	
	\begin{proof}
		Throughout, write $K=(\xi+1)(r-\xi)$ and $L=(\eta+1)(r-\eta)$ for points of
		$\mathcal F(r,t)$.
		
		$(i)$ We have $K=(c+p)(c-p)=c^{2}-p^{2}$ and likewise $L=c^{2}-q^{2}$, so
		$K+L=2c^{2}-p^{2}-q^{2}$ and $K+L=rt$ becomes $p^{2}+q^{2}=2c^{2}-rt=R^{2}$. Since
		$t<\frac{r+2}{2}$ we get $R^{2}>\frac{(r+1)^{2}}{2}-\frac{r(r+2)}{2}=\frac12$.
		
		$(ii)$ With $x-\xi=y-p$ and $x-\eta=y-q$, expanding~\eqref{eq:quad} and using
		$2c^{2}-p^{2}-q^{2}=rt$ together with $pq=\frac{\sigma^{2}-R^{2}}{2}$ produces the
		displayed quadratic; every coefficient depends on $(p,q)$ only through $\sigma$.
		
		$(iii)$ Order $p\le q$. Then $q\ge\frac{R}{\sqrt2}>\frac12$ by $(i)$, while
		$p\ge(r-t)-\frac{r-1}{2}=c-t>-\frac12$ because $t<\frac{r+2}{2}$; hence $\sigma>0$.
		Expanding,
		\[
		\Delta(\sigma)=(r-\sigma)^{2}-2(\sigma^{2}-R^{2})-\frac2t\bigl(\sigma^{3}-R^{2}\sigma\bigr)
		+\frac{4c^{2}\sigma}{t},
		\]
		so $\Delta''(\sigma)=-2-\frac{12\sigma}{t}<0$. Also $\Delta>0$, since $t^{2}\Delta$ is the discriminant
		of~\eqref{eq:quad} and a direct computation gives
		$\Delta=\bigl(\xi-\eta-\frac{K-L}{t}\bigr)^{2}+\frac{4KL}{t^{2}}$. Since $u\mapsto\sqrt u$ is strictly concave and
		increasing, $\sqrt\Delta$ is strictly concave, and adding an affine function preserves
		strict concavity. A strictly concave function on an interval attains its minimum only at
		an endpoint.
		
		It remains to identify the endpoints. Since $p\ge c-t\ge0$ and $p\le q$, the arc lies in
		the part of the circle between the angles $\frac{\pi}{4}$ and $\frac{\pi}{2}$, where
		$\sigma=\sqrt2R\cos\bigl(\varphi-\frac{\pi}{4}\bigr)$ is strictly monotone in the angle
		$\varphi$; thus $\sigma$ parametrizes the arc, and its endpoints correspond to the endpoints
		of the $\sigma$-interval. Moving along the arc from $p=q$ toward smaller $p$,
		the first constraint met is $p\ge c-t$: at $p=c-t$ we have
		$q^{2}=R^{2}-(c-t)^{2}=c^{2}-t(r+t-2c)=c^{2}-t(t-1)\le c^{2}$, so $q\le c$ and the
		constraint $\eta\le r$ is not active there. Hence the endpoints are $\sigma=\sqrt2R$,
		where $p=q$, and the point where $\xi=r-t$.
	\end{proof}
	
	\begin{proof}[Proof of Proposition~\ref{prop:arc}]
	By Lemma~\ref{lem:key}, $\lambda_1(\Qt)$ is the largest root of
	\[
	g(x)=t(x-\xi)(x-\eta)+(x-\xi)L+(x-\eta)K ,
	\]
	which involves $a$ and $b$ only through $K=a(r-\xi)$ and $L=b(r-\eta)$. We therefore
	record the constraints in the coordinates $(\xi,\eta,K,L)$. Since $G$ is connected,
	$K$ and $L$ are positive, so $\xi<r$ and $\eta<r$. Every vertex of $A$ has at most
	$t$ neighbors in $S$, so $K\le at$, which reads
	\begin{equation}\label{eq:c1}
		\xi\ge r-t ;
	\end{equation}
	and $\xi\le a-1$ gives $a\ge\xi+1$, which reads
	\begin{equation}\label{eq:c2}
		K\ge(\xi+1)(r-\xi),
	\end{equation}
	with equality exactly when $G[A]=K_{a}$. The same two statements hold for $B$ with
	$(\eta,L,b)$ in place of $(\xi,K,a)$. Finally
	\begin{equation}\label{eq:c3}
		K+L\le rt ,
	\end{equation}
	with equality exactly when $S$ is independent.
	
	Since $t$ is an integer, $t<\frac{r+2}{2}$ is the same as $t\le\frac{r+1}{2}$, so
	$r-t\ge\frac{r-1}{2}$ and, by~\eqref{eq:c1}, both $\xi$ and $\eta$ lie in the interval
	$[\frac{r-1}{2},r]$ on which $u\mapsto(u+1)(r-u)$ is decreasing. This is how the hypothesis
	$t<\frac{r+2}{2}$ enters here; equivalently, it is what prevents
	$\xi=\eta=r-t$, since that case would force $2t(r+1-t)\le rt$ by~\eqref{eq:c2}
	and~\eqref{eq:c3}, and it is why the two regimes of Theorem~\ref{thm:main} meet at
	$\frac{r+2}{2}$. Relabel $A$ and $B$, if necessary, so that
	$\xi\le\eta$.
	
	\medskip
	\noindent\emph{Claim.} Let $\xi<\eta$ and $K,L>0$, and let $\lambda$ be the largest root of
	$g$. Then $\xi<\lambda<\eta$, and $\lambda$ is strictly increasing in $K$, strictly
	decreasing in $L$, and strictly increasing in $\eta$.
	
	\smallskip
	\noindent\emph{Proof of the Claim.} We have $g(\xi)=(\xi-\eta)K<0$ while
	$g(\eta)=(\eta-\xi)L>0$, so the upward parabola $g$ has a root in $(\xi,\eta)$ and stays
	positive from $\eta$ on. Hence its largest root $\lambda$ lies in $(\xi,\eta)$ and
	$g'(\lambda)>0$. Now
	\[
	\frac{\partial g}{\partial K}\Big|_{\lambda}=\lambda-\eta<0,\qquad
	\frac{\partial g}{\partial L}\Big|_{\lambda}=\lambda-\xi>0,\qquad
	\frac{\partial g}{\partial\eta}\Big|_{\lambda}=-t(\lambda-\xi)-K<0,
	\]
	and in each case the derivative of $\lambda$ is minus the displayed quantity divided by
	$g'(\lambda)$. This proves the Claim.
	
	\medskip
	Suppose first that $\xi=\eta$. Then $\lambda_1(\Qt)=\xi$ by
	Corollary~\ref{cor:diag}, while~\eqref{eq:c2} and~\eqref{eq:c3} give
	$2(\xi+1)(r-\xi)\le rt$. As $2(\psym+1)(r-\psym)=rt$ by~\eqref{eq:psi} and
	$u\mapsto(u+1)(r-u)$ is decreasing on $[\frac{r-1}{2},r]$, this forces
	$\xi\ge\psym(r,t)$. Since $(\psym,\psym)\in\mathcal F(r,t)$ and
	$\Lambda(\psym,\psym)=\psym$, again by Corollary~\ref{cor:diag}, we conclude that
	$\lambda_1(\Qt)\ge\Lambda(\psym,\psym)$.
	
	Suppose now that $\xi<\eta$. We modify $K$, $L$ and $\eta$ in three steps, keeping
	$\xi$ fixed; by the Claim no step increases the largest root of $g$.
	First replace $K$ by $K_{1}=(\xi+1)(r-\xi)$, which does not increase $K$
	by~\eqref{eq:c2}. Next replace $L$ by $L_{1}=rt-K_{1}$, which does not decrease $L$,
	since $L\le rt-K\le rt-K_{1}$ by~\eqref{eq:c3}; here $K_{1}\le t(r+1-t)$
	by~\eqref{eq:c1}, so $L_{1}\ge t(t-1)>0$. Finally put
	\[
	\eta_{1}=\min\bigl\{\eta'\ge\xi:\ (\eta'+1)(r-\eta')\le L_{1}\bigr\},
	\]
	the minimum being attained because $u\mapsto(u+1)(r-u)$ is decreasing on
	$[\xi,r]$, and replace $\eta$ by $\eta_{1}$; this does not increase $\eta$, since
	$(\eta+1)(r-\eta)\le L\le L_{1}$ by~\eqref{eq:c2} for $B$. For the last step
	the Claim applies on $(\xi,\eta]$, and the largest root of $g$ depends
	continuously on $\eta$, so the conclusion persists when $\eta_{1}=\xi$.
	
	If $\eta_{1}=\xi$, then the modified polynomial has equal roots
	$\xi=\eta_{1}$, so its largest root is $\xi$ by Corollary~\ref{cor:diag}, while
	$(\xi+1)(r-\xi)\le L_{1}=rt-(\xi+1)(r-\xi)$ forces $\xi\ge\psym(r,t)$
	exactly as above; hence $\lambda_1(\Qt)\ge\Lambda(\psym,\psym)$.
	
	If $\eta_{1}>\xi$, then $(\eta_{1}+1)(r-\eta_{1})=L_{1}$ by the minimality of
	$\eta_{1}$, so that
	\[
	K_{1}=(\xi+1)(r-\xi),\qquad L_{1}=(\eta_{1}+1)(r-\eta_{1}),\qquad
	K_{1}+L_{1}=rt ,
	\]
	and $r-t\le\xi<\eta_{1}\le r$. This is precisely the condition defining
	$\mathcal F(r,t)$ in~\eqref{eq:F}, and with these values $g$ becomes the left side
	of~\eqref{eq:quad}; hence $\lambda_1(\Qt)\ge\Lambda(\xi,\eta_{1})$.
	
	In every case $\lambda_2(G)\ge\lambda_1(\Qt)\ge\Lambda(\xi',\eta')$ for some
	$(\xi',\eta')\in\mathcal F(r,t)$, by~\eqref{eq:Qt}. Together with
	Lemma~\ref{lem:concave}, this proves the proposition.
	\end{proof}

	\subsection{Deciding between the two endpoints}
	
	\begin{prop}\label{prop:endpoint}
		Let $2\le t<\frac{r+2}{2}$. The value $\psym(r,t)$ of $\Lambda$ at the endpoint $\xi=\eta$
		of $\mathcal F(r,t)$ is smaller than, equal to, or larger than its value
		$\psi_{\mathrm{bd}}(r,t)$ at the endpoint $\xi=r-t$ according as $W(r,t)$ is positive,
		zero, or negative. In particular, the minimum in Proposition~\ref{prop:arc} is
		$\psym(r,t)$ if $W(r,t)\ge0$ and $\psi_{\mathrm{bd}}(r,t)$ if $W(r,t)\le0$.
	\end{prop}
	
	On the diagonal, Corollary~\ref{cor:diag} gives $\Lambda=\xi$, and the budget equation
	$2(\xi+1)(r-\xi)=rt$ yields $\xi=\psym(r,t)$ as in~\eqref{eq:psi}. On the other
	endpoint we have $\xi=r-t$, hence $K=t(r+1-t)$ and $L=rt-K=t(t-1)$, and $\eta$ is
	determined by $(\eta+1)(r-\eta)=t(t-1)$, so that $\eta=\eta_0$.
	
	\begin{proof}[Proof of Proposition~\ref{prop:endpoint}]
		At the endpoint $\xi=r-t$, dividing~\eqref{eq:quad} by $t$ and substituting
		$K=t(r+1-t)$, $L=t(t-1)$ gives~\eqref{eq:psibd}, whose largest root is
		$\psi_{\mathrm{bd}}(r,t)$. Denote the left side of that quadratic by $h$. The product of
		its two roots is $-\eta_0-(t-1)(r-t)<0$, so its smaller
		root is negative while $\xi_0=\psym(r,t)>0$; as $h$ opens upwards, we therefore have
		$\xi_0\le\psi_{\mathrm{bd}}(r,t)$ if and only if $h(\xi_0)\le0$, that is, if and
		only if
		\[
		\Phi(r,t):=(\xi_0+1)\eta_0-\xi_0^{2}-\xi_0t+(t-1)(r-t)\ \ge\ 0 .
		\]
		Write $S_0=\sqrt{(r+1)^{2}-2rt}$ and $D_0=\sqrt{(r+1)^{2}-4t(t-1)}$, so that
		$\xi_0=\frac{r-1+S_0}{2}$ and $\eta_0=\frac{r-1+D_0}{2}$. A direct computation gives
		\[
		4\Phi=D_0X-Y,\qquad X=r+1+S_0,\qquad Y=S_0(r+2t-1)+M,
		\]
		where $M=r^{2}-4rt+4r+4t^{2}-6t+3=(r-2t)^{2}+4r-6t+3$. Since $t<\frac{r+2}{2}$ forces
		$r\ge2t-1$, we get $4r-6t+3\ge2t-1>0$, so $M>0$ and hence $Y>0$ and $D_0X+Y>0$.
		Therefore $\Phi\ge0$ if and only if $D_0^{2}X^{2}\ge Y^{2}$. Reducing $D_0^{2}X^{2}-Y^{2}$
		modulo $S_0^{2}=(r+1)^{2}-2rt$ leaves $A+BS_0$, where
		\[
		\begin{aligned}
			\tfrac14 A=\;&r^{3}t-7r^{2}t^{2}+10r^{2}t-2r^{2}+12rt^{3}-30rt^{2}+23rt-4r\\
			&-4t^{4}+12t^{3}-18t^{2}+12t-2,\\
			\tfrac14 B=\;&r^{2}t-rt+2r-4t^{3}+6t^{2}-4t+2 ,
		\end{aligned}
		\]
		and a further computation gives the identity
		\begin{equation}\label{eq:resultant}
			B^{2}S_0^{2}-A^{2}=16\,t\,(r-2t+2)^{2}\,W(r,t).
		\end{equation}
		Substituting $r=2t-1+u$ with $u\ge0$ gives
		$\frac14B=tu^{2}+(4t^{2}-3t+2)u+2t$, all of whose coefficients are positive, so $B>0$.
		
		We first show that $BS_0>A$. From $S_0^{2}-(r+1-2t)^{2}=2t(r+2-2t)>0$ we get
		$S_0>r+1-2t\ge0$, so it is enough that $B(r+1-2t)\ge A$, that is, that $E\ge0$, where
		\[
		\begin{aligned}
			4E:=B(r+1-2t)-A=\;&5r^{2}t^{2}-10r^{2}t+4r^{2}-16rt^{3}+38rt^{2}-32rt+8r\\
			&+12t^{4}-28t^{3}+32t^{2}-20t+4 .
		\end{aligned}
		\]
		Substituting $r=2t-1+u$ gives
		\[
		E=(5t^{2}-10t+4)u^{2}+4t(t^{2}-3t+1)u+t(t-2)(4t-1),
		\]
		whose three coefficients are positive for $t\ge3$; and for $t=2$ it reduces to $4u(u-2)$,
		which is nonnegative except when $u=1$, that is, except for $(r,t)=(4,2)$. In that case
		$\xi_0=3$ while
		$\psi_{\mathrm{bd}}(4,2)=\frac{-1+\sqrt{17}+\sqrt{74+6\sqrt{17}}}{4}=3.2649\ldots>3$, so
		$\Phi(4,2)>0$, and $W(4,2)=304>0$; the two agree. Thus $BS_0>A$ in every case, directly at
		$(r,t)=(4,2)$, where $BS_0=216>88=A$.
		
		Now $r\ge2t-1$ gives $r-2t+2\ge1$, so $16t(r-2t+2)^{2}>0$. Multiplying $A+BS_0$ by the
		positive number $BS_0-A$ therefore preserves its sign, and by~\eqref{eq:resultant}
		\[
		   (A+BS_0)(BS_0-A)=B^{2}S_0^{2}-A^{2}=16t(r-2t+2)^{2}W(r,t) .
		\]
		Hence $A+BS_0$ and $W(r,t)$ have the same sign, so $\Phi\ge0$ if and only if $W(r,t)\ge0$,
		with $\Phi>0$ if and only if $W(r,t)>0$ and $\Phi=0$ exactly when $W(r,t)=0$. Since
		$\xi_0<\psi_{\mathrm{bd}}(r,t)$,
		$\xi_0=\psi_{\mathrm{bd}}(r,t)$ and $\xi_0>\psi_{\mathrm{bd}}(r,t)$ correspond to $\Phi>0$,
		$\Phi=0$ and $\Phi<0$, this gives the trichotomy in the statement, and combining it with
		Lemma~\ref{lem:concave} gives the last assertion.
	\end{proof}
	
	\begin{prop}\label{prop:nonzero}
		$W(r,t)\ne0$ for all integers $r$ and $t$ with $2\le t<r$.
	\end{prop}
	
	\begin{proof}
		Put $x=t-1$, so that $x\ge1$, and set $s=3x+1$. In terms of $x$,
		\[
		   W=4x(3x+1)r^{3}-(3x+1)(11x^{2}+1)r^{2}+4x(3x+1)(x^{2}-x+1)r-4x^{3}(x^{2}+x+2).
		\]
		Every term except the last is divisible by $s$, so $W\equiv-4x^{3}(x^{2}+x+2)\pmod s$. From
		$3x\equiv-1\pmod s$ we get $27x^{3}\equiv-1$ and
		$9(x^{2}+x+2)=(3x)^{2}+3(3x)+18\equiv1-3+18\equiv16\pmod s$, whence
		\[
		   243\,W\equiv-4\cdot27x^{3}\cdot9(x^{2}+x+2)\equiv-4\cdot(-1)\cdot16\equiv64\pmod s .
		\]
		Suppose $W=0$. Then $s\mid64$, and since $s\equiv1\pmod3$ and $s\ge4$, the only
		possibilities are $s\in\{4,16,64\}$, that is, $x\in\{1,5,21\}$. We rule out each.
	
		If $x=1$, then $W=16(r^{3}-3r^{2}+r-1)$. An integer root of $r^{3}-3r^{2}+r-1$ divides the
		constant term, so $r\in\{1,-1\}$, contradicting $r>t=2$.
	
		If $x=5$, then $W=64(5r^{3}-69r^{2}+105r-250)$. Reducing modulo $5$ gives
		$W\equiv 64\,r^{2}\pmod 5$, so $5\mid r$. For $r\ge15$,
		\[
		   W>64\bigl(5r^{3}-70r^{2}+105r-1470\bigr)=320\,(r^{2}+21)(r-14)>0 ,
		\]
		the first inequality because the difference of the two cubics is $64(r^{2}+1220)>0$. For
		$2<r\le10$,
		\[
		   W<64\bigl(5r^{3}-65r^{2}+110r\bigr)=320\,r(r-2)(r-11)<0 ,
		\]
		the first inequality because the difference is $-64(4r^{2}+5r+250)<0$. The remaining
		values $11\le r\le14$ contain no multiple of $5$ and are excluded by $5\mid r$.
	
		If $x=21$, then $W=64\bigl(84r^{3}-4852r^{2}+35364r-268569\bigr)$. The first three terms
		inside the parentheses are even and the last is odd, so the sum is odd and cannot vanish.
	\end{proof}
	
	Consequently the hypotheses $W(r,t)\ge0$ and $W(r,t)>0$ agree throughout the range, the
	middle case of the trichotomy in Proposition~\ref{prop:endpoint} never occurs, and
	parts $(ii)$ and $(iii)$ of Theorem~\ref{thm:main} between them cover all of
	$2\le t<\frac{r+2}{2}$. This proves the last assertion of Theorem~\ref{thm:main}.
	
	\begin{rem}\label{rem:asym}
		Substituting $t=\tau r$ and letting $r\to\infty$ gives
		\[
		\lim_{r\to\infty}\frac{W(r,\tau r)}{r^{5}}=-\tau^{2}\bigl(4\tau^{3}-12\tau^{2}+33\tau-12\bigr),
		\]
		so asymptotically the criterion $W\ge0$ reads $\tau\le\tau^{*}$, where $\tau^{*}$ is the
		unique real root of $4\tau^{3}-12\tau^{2}+33\tau-12$ in $(0,1)$. Substituting
		$\tau=y+1$ depresses this cubic to $y^{3}+\frac{21}{4}y+\frac{13}{4}=0$, and since
		$\bigl(\frac{13}{8}\bigr)^{2}+\bigl(\frac{7}{4}\bigr)^{3}=8$, Cardano's formula gives
		\[
		\tau^{*}=1+\frac{\sqrt[3]{16\sqrt2-13}-\sqrt[3]{16\sqrt2+13}}{2}=0.4184209\ldots
		\]
		Thus the diagonal governs the bound for $t$ up to roughly $0.4184\,r$, and the other
		endpoint governs it from there up to $\frac{r+2}{2}$.
		
		The next term can be read off in the same way. Writing
		$W(r,\tau r)=P(\tau)r^{5}+Q(\tau)r^{4}+O(r^{3})$ with
		$P(\tau)=-\tau^{2}(4\tau^{3}-12\tau^{2}+33\tau-12)$ and
		$Q(\tau)=-20\tau+88\tau^{2}-56\tau^{3}+16\tau^{4}$, and putting $t=\tau^{*}r+c$, we get
		$W=\bigl[P'(\tau^{*})c+Q(\tau^{*})\bigr]r^{4}+O(r^{3})$. Hence the largest real $t$ with
		$W(r,t)\ge0$ equals $\tau^{*}r+c^{*}+O(r^{-1})$, where
		\[
		c^{*}=-\frac{Q(\tau^{*})}{P'(\tau^{*})}=0.780998\ldots
		\]
		is the real root of $81z^{3}-189z^{2}+216z-92$.
	\end{rem}
	
	The graphs attaining $\psi_{\mathrm{bd}}(r,t)$ are built as follows. Let
	$2\le t<\frac{r+2}{2}$ and suppose that
	\begin{equation}\label{eq:square}
		(r+1)^{2}-4t(t-1)=D^{2}
	\end{equation}
	for an integer $D$ with $1\le D\le r-1$. Put
	\[
	a=r+1-t,\qquad b=\frac{r+1+D}{2},
	\]
	and let $H(r,t)$ have $A\cong K_a$ joined to every vertex of an independent set $S$ of
	size $t$, and $B\cong K_b$ meeting $S$ in a bipartite graph biregular of degrees
	$\bigl(r+1-b,\,t-1\bigr)$; see Figure~\ref{fig:bd}.
	
	\begin{prop}\label{prop:bd}
		If~\eqref{eq:square} holds with $1\le D\le r-1$, then the graph $H(r,t)$ is a connected
		$r$-regular graph on $a+b+t$ vertices with $\kappa\bigl(H(r,t)\bigr)=t$ and
		$\lambda_2\bigl(H(r,t)\bigr)=\psi_{\mathrm{bd}}(r,t)$. Consequently, if in addition
		$W(r,t)\le0$, then $\psi_{\mathrm{bd}}(r,t)$ is a sharp threshold for $(r,t)$.
	\end{prop}
	
	\begin{proof}
		Since $D^{2}\equiv(r+1)^{2}\pmod4$ by~\eqref{eq:square}, we have $D\equiv r+1\pmod2$ and
		$b$ is an integer. As $b$ and $r+1-b$ are the two roots of $z^{2}-(r+1)z+t(t-1)$,
		\begin{equation}\label{eq:bkb}
			b\,(r+1-b)=t(t-1).
		\end{equation}
		Write $k=r+1-b$. From
		$D^{2}-(r+1-2t)^{2}=4t(r+2-2t)>0$ and $r+1-2t\ge0$ we get $D>r+1-2t$, hence
		\begin{equation}\label{eq:bgta}
			b>r+1-t=a\ge t ,
		\end{equation}
		so $k\le t-1$; and $D\le r-1$ gives $k\ge1$. Since also $t-1\le b$, a bipartite graph
		between $B$ and $S$ that is biregular of degrees $(k,t-1)$ exists by~\eqref{eq:bkb}.
		
		A vertex of $A$ has $(a-1)+t=r$ neighbors, a vertex of $B$ has $(b-1)+k=r$, and a vertex
		of $S$ has $a+(t-1)=r$. Thus $H(r,t)$ is $r$-regular, and it is a clique pair over $S$
		with
		\[
		\xi=a-1=r-t,\qquad \eta=b-1,\qquad K=at=t(r+1-t),\qquad L=bk=t(t-1).
		\]
		By the computation preceding the proof of Proposition~\ref{prop:endpoint}, the largest root of the
		corresponding polynomial~\eqref{eq:key} is $\lambda:=\psi_{\mathrm{bd}}(r,t)$, so
		$\lambda_1(\Qt)=\lambda$.
		
		We check that $\kappa(H(r,t))=t$. Deleting $S$ separates $A$ from $B$, so $\kappa\le t$.
		Let $X\subseteq V(H(r,t))$ with $|X|\le t-1$. Then $S\setminus X\ne\emptyset$, and
		$A\setminus X\ne\emptyset$ because $|A|=r+1-t>t-1$, and $B\setminus X\ne\emptyset$
		by~\eqref{eq:bgta}. Every vertex of $S\setminus X$ is adjacent to every vertex of
		$A\setminus X$, so $(A\cup S)\setminus X$ is connected. Each vertex of $X$ meets at most
		$\max\{k,t-1\}=t-1$ edges of the bipartite graph between $B$ and $S$, so at least
		$t(t-1)-(t-1)^{2}=t-1>0$ of those edges survive, and $B\setminus X$ is joined to
		$S\setminus X$. Hence $H(r,t)-X$ is connected and $\kappa=t$.
		
		For the eigenvalue we apply Lemma~\ref{lem:spec}. Every vertex of $A$ is adjacent to every
		vertex of $S$, so only $B$ contributes to $\nu$, and $\nu\le k(t-1)$. From
		$\lambda^{2}=(\eta-t)\lambda+\eta+(t-1)(r-t)$ we get
		\[
		\lambda(\lambda+1)=(\eta-t+1)\lambda+\eta+(t-1)(r-t) .
		\]
		Here $\lambda>0$, since the constant term $-\eta-(t-1)(r-t)$ of that quadratic is
		negative, and $\eta-t+1\ge a-t+1\ge1$ by~\eqref{eq:bgta}. Moreover $b\ge t+1$
		by~\eqref{eq:bgta}, so
		\[
		(t-1)(r-t)-k(t-1)=(t-1)(b-t-1)\ge0 .
		\]
		Therefore $\lambda(\lambda+1)>k(t-1)\ge\nu$, so $\theta<\lambda$ and
		Lemma~\ref{lem:spec} gives $\lambda_2(H(r,t))=\lambda$.
		
		Finally, if $W(r,t)\le0$, then by Proposition~\ref{prop:endpoint} the minimum in Proposition~\ref{prop:arc}
		is the endpoint value $\psi_{\mathrm{bd}}(r,t)$, so every connected $r$-regular graph $G$
		with $\kappa(G)\le t$ satisfies $\lambda_2(G)\ge\psi_{\mathrm{bd}}(r,t)$, and $H(r,t)$
		attains it.
	\end{proof}
	
	\begin{rem}\label{rem:bd}
		The hypothesis $W(r,t)\le0$ in Proposition~\ref{prop:bd} cannot be dropped. For
		$(r,t)=(6,3)$ we have $(r+1)^{2}-4t(t-1)=25$, so $H(6,3)$ is defined with $D=5$, $a=4$
		and $b=6$; it is a $6$-regular graph on $13$ vertices with
		$\lambda_2=1+2\sqrt3=4.4641\ldots$, but $W(6,3)=1508>0$, so the minimum in
		Proposition~\ref{prop:arc} is the diagonal value $\psym(6,3)=\frac{5+\sqrt{13}}{2}=4.3027\ldots$
		and $H(6,3)$ does not attain it.
		
		On the other hand, the hypotheses of Proposition~\ref{prop:bd} hold for infinitely many
		pairs. Taking
		$r=m(m+1)$ and $t=\frac r2$ with $m\ge3$ and writing $u=m^{2}+m$, we get
		$(r+1)^{2}-4t(t-1)=(u+1)^{2}-u(u-2)=4u+1=(2m+1)^{2}$, so~\eqref{eq:square} holds with
		$D=2m+1\le r-1$. For these pairs
		\[
		W\Bigl(r,\frac r2\Bigr)=-\tfrac12\bigl(r^{5}-12r^{4}+20r^{3}+16r^{2}-20r-16\bigr)
		=-\tfrac12\Bigl(r^{3}\bigl(r^{2}-12r+20\bigr)+16r^{2}-20r-16\Bigr),
		\]
		and both $r^{2}-12r+20$ and $16r^{2}-20r-16$ are positive for $r\ge11$, so
		$W(r,\frac r2)<0$ because $r=m(m+1)\ge12$. The case $m=3$ is $(r,t)=(12,6)$, with $D=7$:
		there $H(12,6)$ has $A\cong K_7$ joined to every vertex of an independent set $S$ of size
		$6$, and $B\cong K_{10}$ meeting $S$ in a bipartite graph biregular of degrees $(3,5)$, so
		$H(12,6)$ is $12$-regular on $23$ vertices with $\kappa=6$ and
		$\lambda_2=\frac{3+\sqrt{165}}{2}$.
		Unlike the pairs of Theorem~\ref{thm:main}$(ii)$, the pairs satisfying~\eqref{eq:square}
		need not have $r$ even; for instance
		$(r,t)=(7,4)$, $(17,8)$ and $(31,16)$ also satisfy~\eqref{eq:square} with $W(r,t)<0$.
	\end{rem}
	
	\begin{figure}[H]
		\centering
		\begin{tikzpicture}[scale=0.92]
			\draw[clq] (-4.9,0) ellipse (1.5 and 2.0);
			\node[align=center] at (-4.9,0) {$K_{a}$};
			\node[tag] at (-4.9,2.35) {$\xi=r-t$};
			\foreach \i/\y in {1/1.35,2/0.45,3/-0.45,4/-1.35} \node[vtx] (l\i) at (-3.55,\y) {};
			\foreach \i/\y in {1/1.6,2/0.8,3/0.0,4/-0.8,5/-1.6} \node[cut] (s\i) at (0,\y) {};
			\draw[blob] (-0.8,2.05) rectangle (0.8,-2.05);
			\node[tag] at (0,2.4) {$S$ independent};
			\draw[clq] (4.9,0) ellipse (1.7 and 2.2);
			\node[align=center] at (4.9,0) {$K_b$};
			\node[tag] at (4.9,2.55) {$\eta>\xi$};
			\foreach \i/\y in {1/1.35,2/0.45,3/-0.45,4/-1.35} \node[vtx] (r\i) at (3.4,\y) {};
			\foreach \i in {1,2,3,4} \foreach \j in {1,2,3,4,5} \draw[se,opacity=0.45] (l\i)--(s\j);
			\foreach \i/\j in {1/1,1/2,2/2,2/3,3/3,3/4,4/4,4/5} \draw[se] (r\i)--(s\j);
			\node[lbl,align=center] at (-1.95,-2.5) {joined to all of $S$};
			\node[lbl,align=center] at ( 2.0,-2.5) {biregular to $S$};
		\end{tikzpicture}
		\caption{The graph $H(r,t)$ of Proposition~\ref{prop:bd}, shown schematically; for
			$(r,t)=(12,6)$ one has $a=7$, $b=10$ and order $23$. The two sides are complete of
			different orders, so $\xi\ne\eta$ and $\lambda_2$ lies strictly between
			them.}\label{fig:bd}
	\end{figure}
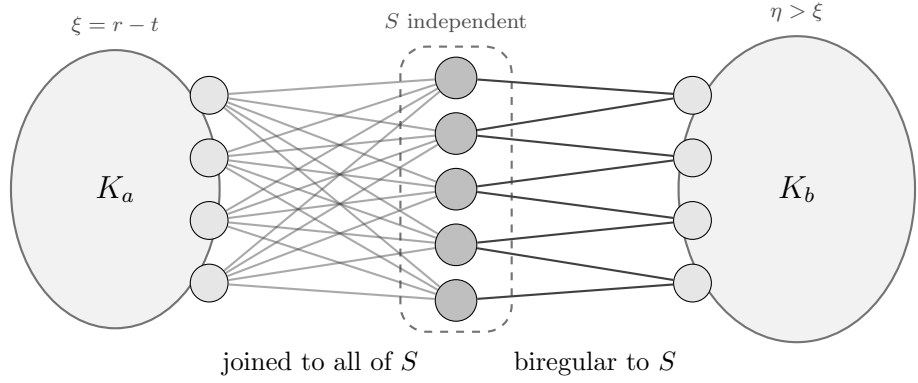
	
	\subsection{The equality cases}
	
	\begin{proof}[Proof of Theorem~\ref{thm:main}$(ii)$]
		Since $W(r,t)>0$, Propositions~\ref{prop:arc} and~\ref{prop:endpoint} show that every
		connected $r$-regular graph $G$ with $\kappa(G)\le t$ satisfies~\eqref{eq:main}, so
		$\psym(r,t)$ is a threshold for $(r,t)$. It remains to determine when it is sharp. We
		first record a consequence of the criterion of Proposition~\ref{prop:endpoint}, which will
		be used in the second half of the proof.
		
		\medskip
		\noindent\emph{Claim.} Let $r$ be even with $r\ge4$ and let $k$ be an integer with
		$1\le k\le\frac r2$ and $t=\frac{2k(r+1-k)}{r}$. Then $3k\le r$.
		
		\smallskip
		\noindent\emph{Proof of the Claim.}
		For $r=4$ the hypotheses force $k=1$, so assume $r\ge6$ and suppose $3k>r$. The function
		$k\mapsto\frac{2k(r+1-k)}{r}$ is increasing on $\bigl[0,\frac{r+1}{2}\bigr]$ and
		$k\le\frac r2$, so $t>\frac{4r+6}{9}$; write $9t=4r+6+v$ with $v>0$. Substituting and
		multiplying by $3^{10}$ gives
		\[
		\begin{aligned}
			-3^{10}W(r,t)=\;&v^{3}\bigl(4v^{2}-4(7r+6)v+1585r^{2}+1560r+576\bigr)\\
			&+2\bigl(7760r^{3}+2493r^{2}-5292r-2484\bigr)v^{2}\\
			&+4\bigl(8948r^{4}-3723r^{3}-17145r^{2}-459r+3807\bigr)v\\
			&+16\bigl(8r^{2}-39r-36\bigr)\bigl(59r^{3}+27r^{2}-135r+27\bigr).
		\end{aligned}
		\]
		The quadratic in $v$ on the first line has discriminant $-192(128r^{2}+123r+45)<0$, so it
		is positive. On the remaining three lines, the coefficients of $v^{2}$ and of $v$ are
		positive for $r\ge2$, and in the constant term the factor $59r^{3}+27r^{2}-135r+27$ is
		positive for $r\ge2$ while $8r^{2}-39r-36$ is positive exactly for $r>\frac{39+9\sqrt{33}}{16}=5.668\ldots$;
		since $r$ is even and $r\ge6$, both factors are positive. Hence $W(r,t)<0$, contrary to
		the hypothesis of the theorem. This proves the Claim.
		
		\medskip
		Suppose now that $\psym(r,t)$ is sharp, so that some connected $r$-regular graph $G$ with
		$\kappa(G)\le t$ satisfies $\lambda_2(G)=\psym(r,t)$, and run the reduction in the proof of
		Proposition~\ref{prop:arc} for a vertex cut $S$ of $G$ with $|S|=t$. Since
		$W(r,t)>0$, Propositions~\ref{prop:arc} and~\ref{prop:endpoint} give
		$\lambda_2(G)\ge\lambda_1(\Qt)\ge\psym(r,t)=\lambda_2(G)$, so every inequality there is an
		equality.
		
		We first check that $\xi=\eta$. Suppose not. If $\eta_1=\xi$, then the last step of the
		reduction replaced $\eta$ by $\eta_1<\eta$, which strictly decreases the largest root by
		the Claim in the proof of Proposition~\ref{prop:arc}; hence
		$\lambda_1(\Qt)>\xi\ge\psym(r,t)$, a contradiction. Otherwise the reduction ends with
		$\lambda_1(\Qt)=\Lambda(\xi,\eta_1)$ for a point $(\xi,\eta_1)\in\mathcal F(r,t)$
		with $\xi<\eta_1$, and this value equals the minimum of $\Lambda$ over
		$\mathcal F(r,t)$. By Lemma~\ref{lem:concave} that minimum is attained only at an
		endpoint of the $\sigma$-interval, and since $W(r,t)>0$ the diagonal endpoint is the
		strictly smaller of the two; hence $\xi=\eta_1$, a contradiction.
		
		So $\xi=\eta$, and $\lambda_1(\Qt)=\xi$ by Corollary~\ref{cor:diag}, whence
		$\xi=\psym(r,t)$. Equality throughout $2(\xi+1)(r-\xi)\le K+L\le rt$ forces
		$K=(\xi+1)(r-\xi)$ and $L=(\eta+1)(r-\eta)$, so $G[A]$ and $G[B]$ are complete,
		and $K+L=rt$, so $S$ is independent. Thus $a=b=\xi+1$ and $K=L=\frac{rt}{2}$, and
		counting the edges between $A$ and $S$ in two ways gives
		\begin{equation}\label{eq:pellk}
			a(r+1-a)=\frac{rt}{2}.
		\end{equation}
		
		Next we show that $r$ is even. Since $G$ is $r$-regular and connected, the all-ones vector
		spans the eigenspace of $\lambda_1(G)=r$, so
		$\lambda_2(G)=\max\bigl\{\frac{v^{\top}A(G)v}{v^{\top}v}:v\ne0,\ v\perp\mathbf 1\bigr\}$,
		the maximum being attained exactly at the eigenvectors for $\lambda_2(G)$ that lie in
		$\mathbf 1^{\perp}$. Let $v$ be equal to $1$ on $A$, to $-1$ on $B$ and to $0$ on $S$.
		Then $v\perp\mathbf 1$, and as $G[A]$ and $G[B]$ are complete of order $a$ while every
		other edge of $G$ meets $S$,
		\[
		\frac{v^{\top}A(G)v}{v^{\top}v}=\frac{2\cdot2\binom a2}{2a}=a-1=\psym(r,t)=\lambda_2(G).
		\]
		Hence $v$ is an eigenvector of $A(G)$ for $a-1$. Reading that equation at a vertex
		$s\in S$ gives $\deg_A(s)-\deg_B(s)=(a-1)\cdot0=0$, while $\deg_A(s)+\deg_B(s)=r$
		because $S$ is independent. Therefore $\deg_A(s)=\deg_B(s)=\frac r2$ for every $s\in S$,
		so $r$ is even, and the bipartite graph between each side and $S$ is biregular of degrees
		$\bigl(r+1-a,\frac r2\bigr)$.
		
		Setting $k=r+1-a$, equation~\eqref{eq:pellk} reads $t=\frac{2k(r+1-k)}{r}$, and since
		$2k(r+1-k)\equiv2k(1-k)\pmod r$ this is an integer exactly when $r\mid2k(k-1)$. Finally
		every vertex of $A$ has $k$ neighbors in $S$, so $k\le t<\frac{r+2}{2}$; as $r$ is even
		and $k$ is an integer, $k\le\frac r2$.
		
		Conversely, given such a $k$, the data above determine a graph $G(r,t)$: take
		$A\cong B\cong K_a$ with $a=r+1-k$, an independent set $S$ of size $t$, and between each
		clique and $S$ a bipartite graph biregular of degrees $\bigl(k,\frac r2\bigr)$, which
		exists by~\eqref{eq:pellk} because $k\le t$ and $\frac r2\le a$. Then $G(r,t)$ is a
		connected $r$-regular clique pair over $S$ with $\xi=\eta=a-1$, so
		$\lambda_1(\Qt)=a-1$ by Corollary~\ref{cor:diag}. In Lemma~\ref{lem:spec} we have
		$k_A=k_B=k$ and $d_A=d_B=\frac r2$, so $\nu\le rk$, and $rk\le a(a-1)$ implies
		$\theta\le a-1$. The Claim gives $k\le\frac r3$, whence
		\[
		a(a-1)=(r+1-k)(r-k)\ge\Bigl(\frac{2r}{3}+1\Bigr)\frac{2r}{3}>\frac{r^{2}}{3}\ge rk .
		\]
		Therefore $\lambda_2(G(r,t))=a-1=\psym(r,t)$ by Lemma~\ref{lem:spec}. Deleting $S$
		separates $A$ from $B$, so $\kappa(G(r,t))\le t$, and hence $G(r,t)$ witnesses the
		sharpness of $\psym(r,t)$. See Figure~\ref{fig:diag}.
	\end{proof}
	
	We turn to the other endpoint, where the answer needs no divisibility and no parity.
	
	\begin{proof}[Proof of Theorem~\ref{thm:main}$(iii)$]
		Since $W(r,t)<0$, Proposition~\ref{prop:endpoint} says that the minimum in
		Proposition~\ref{prop:arc} is the
		endpoint value $\psi_{\mathrm{bd}}(r,t)$, and that it is strictly smaller than
		$\psym(r,t)$; in particular $\psi_{\mathrm{bd}}(r,t)$ is a threshold for $(r,t)$.
		If~\eqref{eq:square} holds with $1\le D\le r-1$, then Proposition~\ref{prop:bd} produces a
		graph attaining it.
		
		Conversely, suppose $\lambda_2(G)=\psi_{\mathrm{bd}}(r,t)$, and run the reduction in the proof of
		Proposition~\ref{prop:arc} for a vertex cut $S$ of $G$ with $|S|=t$. Every inequality there is
		an equality. If $\xi=\eta$, or if $\eta_{1}=\xi$, the reduction ends with largest root
		$\xi\ge\psym(r,t)>\psi_{\mathrm{bd}}(r,t)$, which is impossible; so $\xi<\eta_{1}$ and
		$\lambda_1(\Qt)=\Lambda(\xi,\eta_{1})$ equals the minimum of $\Lambda$ over
		$\mathcal F(r,t)$. By Lemma~\ref{lem:concave} that minimum is attained only at an endpoint
		of the $\sigma$-interval, and the diagonal endpoint has the larger value $\psym(r,t)$;
		hence $\xi=r-t$.
		
		By the strict monotonicity in the Claim in the proof of Proposition~\ref{prop:arc}, the three
		steps of the reduction were
		equalities: $K=(\xi+1)(r-\xi)$, so $G[A]=K_{a}$; then $L=rt-K$, so $S$ is independent; and
		$\eta=\eta_{1}$, so $G[B]=K_{b}$.
		
		Now $\xi=r-t$ gives $a=r+1-t$ and $K=at$, so every vertex of $A$ is adjacent to every
		vertex of $S$. Hence $L=rt-K=t(t-1)$, while $G[B]=K_{b}$ gives $L=b(r+1-b)$, so
		\[
		b(r+1-b)=t(t-1)\qquad\text{and}\qquad (r+1-2b)^{2}=(r+1)^{2}-4t(t-1).
		\]
		Thus~\eqref{eq:square} holds with $D=2b-(r+1)$, which is positive because
		$b>\xi+1=r+1-t\ge\frac{r+1}{2}$, and at most $r-1$ because $G$ is connected and therefore
		$r+1-b\ge1$. Finally every vertex of $B$ has $r+1-b$ neighbors in $S$, and every vertex of
		$S$ has $a$ neighbors in $A$ and $r-a=t-1$ in $B$, so the bipartite graph between $B$ and
		$S$ is biregular of degrees $(r+1-b,t-1)$. Hence $G$ is one of the graphs $H(r,t)$ of
		Proposition~\ref{prop:bd}.
	\end{proof}
	
	\begin{figure}[H]
		\centering
		\begin{tikzpicture}[scale=0.92]
			\draw[clq] (-5.2,0) ellipse (1.75 and 2.15);
			\draw[clq] ( 5.2,0) ellipse (1.75 and 2.15);
			\node[align=center] at (-5.2,0) {$K_a$\\[3pt] \footnotesize $a=r+1-k$};
			\node[align=center] at ( 5.2,0) {$K_a$};
			\foreach \i/\y in {1/1.35,2/0.45,3/-0.45,4/-1.35}
			{ \node[vtx] (l\i) at (-3.65,\y) {}; \node[vtx] (r\i) at (3.65,\y) {}; }
			\draw[blob] (-0.8,2.0) rectangle (0.8,-2.3);
			\node[cut] (s1) at (0, 1.45) {};
			\node[cut] (s2) at (0, 0.55) {};
			\node[cut] (s3) at (0,-1.75) {};
			\node[tag] at (0,-0.55) {$\vdots$};
			\node[tag] at (1.15,1.45) {$s_1$};
			\node[tag] at (1.15,-1.75) {$s_t$};
			\foreach \i in {1,2,3,4} \foreach \j in {1,2,3}
			{ \draw[se,opacity=0.6] (l\i)--(s\j); \draw[se,opacity=0.6] (r\i)--(s\j); }
			\node[lbl,align=center] at (0,2.75)
			{$S$ independent, $|S|=t$; each $s\in S$ sends $\frac r2$ to each side};
			\node[lbl,align=center] at (0,-3.05) {$\xi=\eta=a-1$, and $a(r+1-a)=\frac{rt}{2}$};
		\end{tikzpicture}
		\caption{The extremal graph $G(r,t)$ of Theorem~\ref{thm:main}$(ii)$. The two sides are complete of
			equal order, so $\lambda_2=a-1=\psym(r,t)$ by Corollary~\ref{cor:diag}.}\label{fig:diag}
	\end{figure}
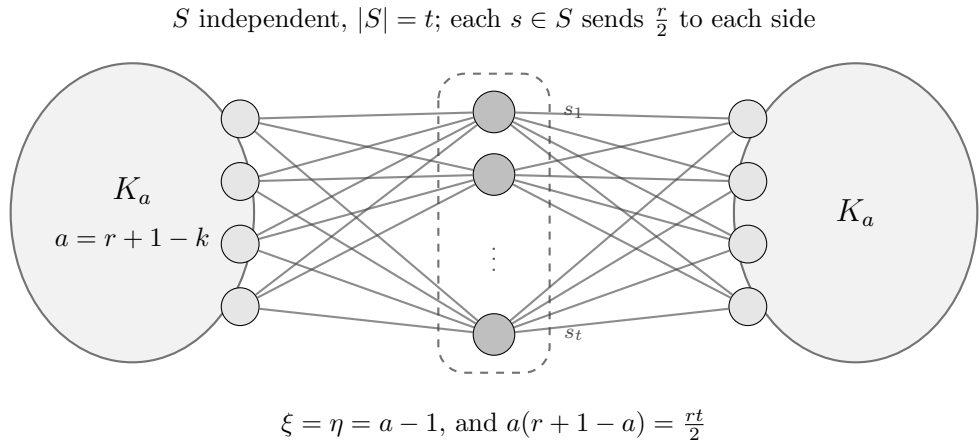
	
	\begin{rem}\label{rem:list}
		Taking $k=1$ gives $2k(k-1)=0$, so every even $r$ qualifies, with $t=2$ and $a=r$; this
		is the family of Figure~\ref{fig:t2}. For $2\le k\le6$ the pairs in the range
		$2\le t<\frac{r+2}{2}$ satisfying the arithmetic conditions of Theorem~\ref{thm:main}$(ii)$
		are $(12,5)$, $(12,6)$, $(24,7)$, $(20,8)$, $(40,9)$, $(20,9)$, $(30,10)$ and $(60,11)$, and
		$W(r,t)>0$ for all of them except $(12,6)$, where $W(12,6)=-18304$. At $(12,6)$ the graph
		$G(12,6)$ of the proof of Theorem~\ref{thm:main}$(ii)$ still exists: it consists of two
		copies of $K_9$ and an independent cut of size $6$, meeting each clique in a bipartite
		graph biregular of degrees $(4,6)$. It is $12$-regular on $24$ vertices with $\kappa=6$,
		and Corollary~\ref{cor:diag} and Lemma~\ref{lem:spec}, with $\nu\le rk=48\le72=a(a-1)$,
		give $\lambda_2(G(12,6))=8=\psym(12,6)$. Yet $\psym(12,6)$ is not a threshold, since
		$H(12,6)$ of Remark~\ref{rem:bd} has $\kappa=6$ and $\lambda_2=\frac{3+\sqrt{165}}{2}<8$.
		Thus $W(r,t)>0$ is a genuine hypothesis of Theorem~\ref{thm:main}$(ii)$. The Claim in its
		proof shows that this hypothesis forces $k\le\frac r3$, whereas the divisibility conditions
		alone allow $k$ up to $\frac r2$; and $(12,6)$, where $k=\frac r3$ but $W<0$, shows that the
		Claim does not replace the hypothesis. Finally, \eqref{eq:pellk} shows that $t=1$ never
		occurs, in agreement with~\cite{ZW}, and that $t=3$ and $t=4$ never occur inside the range.
	\end{rem}
	
	\subsection{Large vertex-connectivity}\label{sec:large}
	
	In this subsection we prove Theorem~\ref{thm:main}$(i)$. Throughout, $t$ is an integer with
	$\frac{r+2}{2}\le t\le r-1$, and we set
	\[
	a=r+1-t .
	\]
	Since $r-t$ is a threshold for $(r,t)$, as noted after~\eqref{eq:fiedler}, it suffices to
	construct, for each such pair, a connected $r$-regular graph $G$ with $\kappa(G)=t$ and
	$\lambda_2(G)=r-t$.
	
	We treat first the case in which $rt$ is even. Let $H$ be a $(2t-r-2)$-regular graph on $t$
	vertices and set
	\[
	E(r,t)=\bigl(K_{a}\cup K_{a}\bigr)\vee H ;
	\]
	see Figure~\ref{fig:large}. Such an $H$ exists: its degree satisfies $0\le2t-r-2\le t-1$ because
	$\frac{r+2}{2}\le t\le r-1$, and $(2t-r-2)t=2t^{2}-2t-rt$ is even exactly when $rt$ is
	even.
	
	\begin{prop}\label{prop:even}
		Let $rt$ be even. Then $E(r,t)$ is a connected $r$-regular graph on $2a+t$ vertices with
		$\kappa(E(r,t))=t$ and $\lambda_2(E(r,t))=r-t$, for every choice of $H$.
	\end{prop}
	
	\begin{proof}
		A vertex of a clique has $a-1$ neighbors inside its clique and $t$ in $H$, so its degree
		is $a-1+t=r$; a vertex of $H$ has $2t-r-2$ neighbors inside $H$ and $2a$ outside, so its
		degree is $2t-r-2+2(r+1-t)=r$. Deleting $V(H)$ separates the two cliques, so
		$\kappa\le t$. Let $X$ be a set of at most $t-1$ vertices. Some vertex of $H$ survives,
		and every vertex of $H$ is adjacent to every vertex outside $V(H)$; so if some vertex
		outside $V(H)$ survives, then $E(r,t)-X$ is connected. Otherwise $X$ contains all $2a$
		vertices outside $V(H)$, and $H-X$ has $n'\ge t-(t-1-2a)=2a+1$ vertices. If $H-X$ were
		disconnected, a vertex in a smallest component would be adjacent in $\overline H$ to the
		at least $\frac{n'}{2}>a$ vertices of the other components, which is impossible since
		$\overline H$ is $a$-regular, as shown below. Hence $\kappa=t$.
		
		For the eigenvalue we pass to the complement. Since complementation turns a join into a
		disjoint union,
		\[
		\overline{E(r,t)}=\overline{K_{a}\cup K_{a}}\ \cup\ \overline H=K_{a,a}\ \cup\ \overline H ,
		\]
		and $\overline{E(r,t)}$ is $(n-1-r)$-regular with $n=2a+t$, so its degree is
		$2a+t-1-r=a$. Thus both pieces are $a$-regular: the piece $K_{a,a}$ visibly, and
		$\overline H$ because $t-1-(2t-r-2)=r-t+1=a$. Any $a$-regular graph satisfies
		$\lambda_{\min}\ge-a$, and the bipartite piece $K_{a,a}$ attains this value, so
		$\lambda_{\min}(\overline{E(r,t)})=-a$ irrespective of $H$. By~\eqref{eq:comp},
		\[
		\lambda_2\bigl(E(r,t)\bigr)=-1-(-a)=a-1=r-t .
		\]
	\end{proof}
	
	\begin{figure}[H]
		\centering
		\begin{tikzpicture}[scale=1.0]
			\node[lbl] at (-2.9,2.15) {$K_a$};
			\node[lbl] at ( 2.9,2.15) {$K_a$};
			\foreach \i/\y in {1/0.95,2/0.0,3/-0.95} \node[vtx] (a\i) at (-2.9,\y) {};
			\draw[ce] (a1)--(a2) (a2)--(a3); \draw[ce] (a1) to[bend right=45] (a3);
			\draw[blob] (-3.7,1.65) rectangle (-2.1,-1.65);
			\foreach \i/\y in {1/0.95,2/0.0,3/-0.95} \node[vtx] (b\i) at (2.9,\y) {};
			\draw[ce] (b1)--(b2) (b2)--(b3); \draw[ce] (b1) to[bend left=45] (b3);
			\draw[blob] (2.1,1.65) rectangle (3.7,-1.65);
			\foreach \i/\y in {1/1.5,2/0.9,3/0.3,4/-0.3,5/-0.9,6/-1.5} \node[cut] (s\i) at (0,\y) {};
			\draw[blob] (-1.05,2.0) rectangle (1.05,-2.0);
			\foreach \i/\j in {1/2,2/3,3/4,4/5,5/6}
			\draw[ce,black!65,line width=0.7pt] (s\i)--(s\j);
			\draw[ce,black!65,line width=0.7pt] (s1) to[bend left=75] (s6);
			\foreach \i in {a1,a2,a3,b1,b2,b3} \foreach \j in {s1,s2,s3,s4,s5,s6}
			\draw[se,opacity=0.4] (\i)--(\j);
			\node[lbl,align=center] at (0,-2.45)
			{$H$ is $(2t-r-2)$-regular and joined to everything outside itself};
			\node[lbl] at (0,-3.0) {$a=r+1-t$, \quad $\lambda_2=a-1=r-t$};
		\end{tikzpicture}
		\caption{The graph $E(r,t)=(K_a\cup K_a)\vee H$, drawn for
			$r=8$ and $t=6$, so that $a=3$ and $H$ is a $6$-cycle. Its complement is
			$K_{a,a}\cup\overline H$ with both pieces $a$-regular.}\label{fig:large}
	\end{figure}
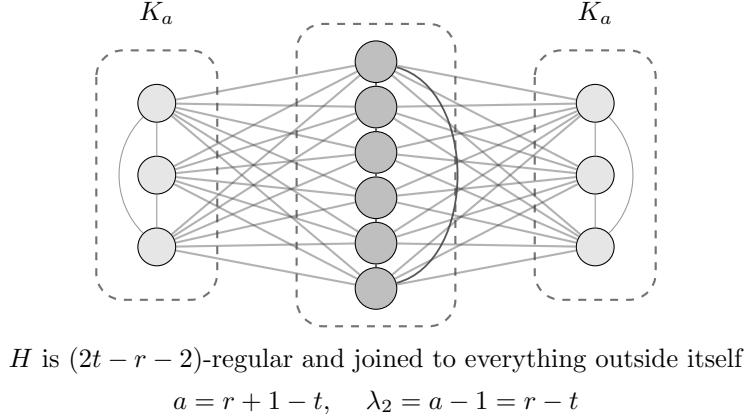
	
	If $r$ and $t$ are both odd, then $2t-r-2$ is odd and $t$ is odd, so no
	$(2t-r-2)$-regular graph on $t$ vertices exists and the graph $E(r,t)$ is
	unavailable. This is not an artifact: the graph $E(r,t)$ would have $2a+t=2r+2-t$
	vertices, an odd number, while $r$ is odd, so no such regular graph exists. We therefore
	give the two sides consecutive orders $a$ and $a+1$. Only the smaller side can be
	complete: a vertex of the larger side is adjacent to all $t$ vertices of the cut, so it
	must have exactly $r-t=a-1$ neighbors inside a set of $a+1$ vertices, and a
	$(a-1)$-regular graph on $a+1$ vertices is the complement of a perfect matching.
	
	Accordingly, for $r$ and $t$ both odd, put $u=r-t$, which is even and positive, let
	\[
	H=\overline{C},\qquad C=\mathrm{circ}\Bigl(t;\,1,2,\dots,\tfrac{u+2}{2}\Bigr),
	\]
	so that $H$ is a $(2t-r-3)$-regular graph on $t$ vertices, let $\widehat H$ be the
	complement of a perfect matching on $u+2$ vertices, a $u$-regular graph, and set
	\[
	E^{*}(r,t)=\bigl(K_{a}\cup\widehat H\bigr)\vee H .
	\]
	
	\begin{prop}\label{prop:odd}
		Let $r$ and $t$ both be odd with $\frac{r+2}{2}\le t<r$. Then $E^{*}(r,t)$ is a connected
		$r$-regular graph on $2r+3-t$ vertices with $\kappa(E^{*}(r,t))=t$ and
		$\lambda_2(E^{*}(r,t))=r-t$.
	\end{prop}
	
	\begin{proof}
		Degrees: a vertex of $K_a$ has $a-1+t=r$ neighbors; a vertex of $\widehat H$ has
		$u+t=r$; a vertex of $H$ has $(2t-r-3)+a+(u+2)=r$. The order is
		$a+t+(u+2)=2r+3-t$, which is even, as it must be. Deleting $V(H)$ separates
		$K_a$ from $\widehat H$, so $\kappa\le t$. Let $X$ be a set of at most $t-1$ vertices.
		As in Proposition~\ref{prop:even}, $E^{*}(r,t)-X$ is connected unless $X$ contains all
		$a+u+2=2u+3$ vertices outside $V(H)$, in which case $H-X$ has $n'\ge2u+4$ vertices.
		Suppose $H-X$ is disconnected. A vertex in a smallest component is adjacent in
		$C=\overline H$ to the at least $\frac{n'}{2}\ge u+2$ vertices of the other components,
		and $C$ is $(u+2)$-regular; so equality holds throughout, $C-X$ is $K_{u+2,u+2}$, and no
		edge of $C$ joins $V(H)\setminus X$ to $X$. As $C$ is connected, $X\cap V(H)=\emptyset$
		and $C=K_{u+2,u+2}$, which is impossible because $C$ contains the triangle on $0,1,2$.
		Hence $\kappa=t$.
		
		The partition $\{V(K_a),V(H),V(\widehat H)\}$ is equitable with quotient matrix
		\[
		Q=\begin{pmatrix}
			r-t & t & 0\\
			a & 2t-r-3 & u+2\\
			0 & t & u
		\end{pmatrix},
		\]
		and a computation gives
		$\det(Q-xI)=(r-t-x)\bigl[(2t-r-3-x)(r-t-x)-t(2r-2t+3)\bigr]$, whose roots are $r-t$, $r$
		and $t-r-3$. Hence $r-t$ is an eigenvalue of $E^{*}(r,t)$. The remaining eigenvalues come
		from inside the three parts, and they are the nontrivial eigenvalues of $K_a$, of $H$ and
		of $\widehat H$. The first are all $-1$. For $\widehat H$, its complement is a perfect
		matching, so $\widehat H$ is a complete multipartite graph with parts of size $2$ and
		$\lambda_2(\widehat H)=0\le r-t$. For $H$, its complement $C$ is $(u+2)$-regular, and we
		claim that $\lambda_{\min}(C)\ge-(u+1)$. Put $s=\frac{u+2}{2}\ge2$. The eigenvalues of $C$
		are $2\sum_{m=1}^{s}\cos(m\theta_j)=D(\theta_j)-1$ for $\theta_j=\frac{2\pi j}{t}$ and
		$0\le j\le t-1$, where
		\[
		D(\theta)=\frac{\sin\bigl((s+\frac12)\theta\bigr)}{\sin(\theta/2)}
		\]
		is the Dirichlet kernel. Since $2s+1$ is odd, $D(2\pi-\theta)=D(\theta)$, so it suffices
		to bound $D$ on $(0,\pi]$. If $\sin(\theta/2)\ge\frac{1}{2s-2}$, then
		$|D(\theta)|\le\frac{1}{\sin(\theta/2)}\le2s-2=u$. Otherwise
		$\theta<2\arcsin\frac{1}{2s-2}$, whence
		$(s+\frac12)\theta<(2s+1)\arcsin\frac{1}{2s-2}\le5\arcsin\frac12=\frac{5\pi}{6}<\pi$,
		the middle quantity being decreasing in $s$; then both $\sin((s+\frac12)\theta)$ and
		$\sin(\theta/2)$ are positive, so $D(\theta)>0$. In either case $D(\theta)\ge-u$, so
		$\lambda_{\min}(C)\ge-u-1$. By~\eqref{eq:comp},
		$\lambda_2(H)=-1-\lambda_{\min}(C)\le u=r-t$. Therefore
		$\lambda_2(E^{*}(r,t))=r-t$.
		
	\end{proof}
	
	\begin{proof}[Proof of Theorem~\ref{thm:main}$(i)$]
		Combine the threshold property of $r-t$ recorded after~\eqref{eq:fiedler} with
		Proposition~\ref{prop:even} when $rt$ is even and with Proposition~\ref{prop:odd} when
		$rt$ is odd.
	\end{proof}
	
	\section{Concluding Remarks}\label{sec:remark}
	
	Theorem~\ref{thm:main} gives a threshold for every pair $(r,t)$ with $t\ge2$ and determines
	exactly when it is sharp, together with all extremal graphs in the sharp cases of $(ii)$
	and $(iii)$. We mention two questions that remain.
	
	The main open problem is to determine the sharp threshold at the pairs $(r,t)$ with
	$2\le t<\frac{r+2}{2}$ for which the threshold of Theorem~\ref{thm:main} is not attained.
	At such a pair the budget $e(A,S)+e(B,S)\le rt$ cannot be divided as the equality analysis
	requires, and the two cases that are known show that this can fail in different ways. For
	$t=1$ the extremal graphs of Zhang and Wang~\cite{ZW} have sides obtained from $K_{r+1}$ by
	deleting a matching, so a side is not complete. For $t=2$ and odd $r\ge5$, in contrast, the
	extremal graph of Zhang~\cite{Z} still consists of two copies of $K_r$ and an independent
	cut; what fails there is the degree condition on the cut vertices, which for $t=2$ would
	require each of them to have $\frac r2$ neighbors in each side. Lemma~\ref{lem:key} is an identity in the average degrees and needs no
	completeness, and the only step of our reduction that uses completeness is the inequality
	$K\ge(\xi+1)(r-\xi)$, where the order of a side is tied to its average degree. Treating the
	orders of the two sides as integer parameters leads to a finer optimization problem, and we
	expect its answer to depend on the arithmetic of $(r,t)$ in a more intricate way than in
	Theorem~\ref{thm:main}.
	
	The second question concerns the transition between $(ii)$ and $(iii)$.
	Remark~\ref{rem:asym} locates the largest $t$ with $W(r,t)>0$ at $\tau^{*}r+c^{*}+O(r^{-1})$.
	In our computations for $r\le400$, the sign of $W(r,t)$ changes at most once as $t$
	increases through $2\le t<\frac{r+2}{2}$: for $r\in\{3,4,6,8\}$ it is positive throughout,
	and for every other $r\le400$ it changes exactly once, from positive to negative. In the
	latter case, $\lfloor\tau^{*}r+c^{*}\rfloor$ gives the largest $t$ with $W(r,t)>0$ except
	when $r=77$. It would be interesting to prove that the sign changes at most once for every
	$r$ and to describe the transition point exactly.
	
	\section*{Acknowledgments}
	
	The authors thank Jianfeng Wang and Wenqian Zhang for helpful discussions, and
	Wenqian Zhang for drawing their attention to~\cite{Z}.
	
	\section*{Declaration of generative AI and AI-assisted technologies in the writing process}
	
	During the preparation of this work, the authors used Claude (Anthropic) for language
	refinement, technical editing, and computational verification of examples. The authors
	reviewed and edited the output as needed and take full responsibility for the content of
	the published article.

\end{document}